\documentclass[12pt]{article}
\usepackage[margin=0.75in]{geometry} 
\usepackage{amsmath,amssymb,amsthm}
\usepackage{color}
\usepackage{graphicx}
\usepackage{subfigure}

\newtheorem{theorem}{Theorem}[section]
\newtheorem{lemma}[theorem]{Lemma}
\newtheorem{proposition}[theorem]{Proposition}
\newtheorem{corollary}[theorem]{Corollary}

\newtheorem{rem}{Remark}

\newtheorem{example}{Example}

\numberwithin{equation}{section}

\newcommand{\be}{\begin{equation}}
\newcommand{\ee}{\end{equation}}

\newcommand{\bea}{\begin{eqnarray}}
\newcommand{\eea}{\end{eqnarray}}
\newcommand{\bean}{\begin{eqnarray*}}
\newcommand{\eean}{\end{eqnarray*}}

\newcommand{\bg}{\begin{gather}}
\newcommand{\eg}{\end{gather}}
\newcommand{\bgn}{\begin{gather*}}
\newcommand{\egn}{\end{gather*}}

\title{Center of Mass Distribution of The Jacobi Unitary Ensembles: Painlev\'{e} \uppercase\expandafter{\romannumeral 5}, Asymptotic Expansions}
\author{Longjun Zhan$^{a,}$\thanks{Zhan\_Longjun@foxmail.com},~
Gordon Blower$^{b,}$\thanks{g.blower@lancaster.ac.uk},~Yang Chen$^a$\thanks{ yangbrookchen@yahoo.co.uk}, ~Mengkun Zhu$^{a,}$\thanks{Corresponding Author: Zhu\_mengkun@163.com}\\
\footnotesize{$^{a}$Department of Mathematics, University of Macau,}\\
\footnotesize{Avenida da Universidade, Taipa, Macau, China}\\
\footnotesize{$^{b}$Department of Mathematics and Statistics, Lancaster University,}\\
\footnotesize{ Lancaster, LA14YF, United Kingdom}}

\begin{document}

\maketitle

\begin{abstract}
In this paper, we study the probability density function, $\mathbb{P}(c,\alpha,\beta, n)\,dc$,  of  the center of mass of the finite $n$ Jacobi unitary ensembles with parameters $\alpha\,>-1$ and $\beta >-1$;
that is  the probability that ${\rm tr}M_n\in(c, c+dc),$ where $M_n$ are $n\times n$ matrices drawn from the unitary Jacobi ensembles.
We first compute the exponential moment generating function of the linear statistics $\sum_{j=1}^{n}\,f(x_j):=\sum_{j=1}^{n}x_j,$  denoted by $\mathcal{M}_f(\lambda,\alpha,\beta,n)$.

The weight function associated with the Jacobi unitary ensembles reads $x^{\alpha}(1-x)^{\beta},\; x\in [0,1]$. The moment generating function is the $n\times n$ Hankel determinant $D_n(\lambda,\alpha,\beta)$
generated by the time-evolved Jacobi weight, namely,
$w(x;\lambda ,\alpha,\beta )=x^{\alpha}(1-x)^{\beta}\,{\rm e}^{-\lambda\:x},\,x\in[0,1],\,\alpha>-1,\,\beta>-1$. We  think of $\lambda$ as the time
variable in the resulting Toda equations. The non-classical polynomials  defined by the monomial expansion,
$P_n(x,\lambda)= x^n+ p(n,\lambda)\:x^{n-1}+\dots+P_n(0,\lambda)$, orthogonal with respect to $w(x,\lambda,\alpha,\beta )$ over $[0,1]$ play an important role.
Taking the time evolution problem studied in Basor, Chen and Ehrhardt \cite{BasorChenEhrhardt2010}, with some change of variables, we obtain a
certain auxiliary variable $r_n(\lambda),$ defined by integral over $[0,1]$ of the product of the unconventional orthogonal polynomials of degree $n$ and $n-1$ and
 $w(x;\lambda ,\alpha,\beta )/x$. It is shown that $r_n(2{\rm i}\/{\rm e}^{z})$ satisfies a Chazy II equation. There is another auxiliary variable, denote as $R_n(\lambda),$ defined by an integral over $[0,1]$ of the product of two polynomials of degree $n$ multiplied by $w(x;\lambda ,\alpha,\beta )/x.$ Then $Y_n(-\lambda)=1-\lambda/R_n(\lambda)$  satisfies a particular Painlev\'{e} \uppercase\expandafter{\romannumeral 5}:
$P_{\rm V}(\alpha^2/2$, $ -\beta^2/2, 2n+\alpha+\beta+1,1/2)$.

The $\sigma_n$ function defined in terms of the $\lambda p(n,-\lambda)$ plus a translation in $\lambda$ is the Jimbo--Miwa--Okamoto
$\sigma-$form of Painlev\'{e} \uppercase\expandafter{\romannumeral 5}. The continuum approximation, treating the collection of eigenvalues as a charged fluid as in the Dyson Coulomb Fluid,
gives an approximation for the moment generation function $\mathcal{M}_f(\lambda,\alpha,\beta,n)$ when $n$ is sufficiently large. Furthermore, we deduce a new expression of $\mathcal{M}_f(\lambda,\alpha,\beta,n)$ when $n$ is finite in terms the $\sigma$ function  of this the Painlev\'{e} \uppercase\expandafter{\romannumeral 5}. An estimate shows that the moment generating function is a function of exponential type and of order $n$. From the Paley-Wiener theorem, one deduces that  $\mathbb{P}(c,\alpha,\beta,n)$ has compact support $[0,n]$. This result  is easily extended to the $\beta$ ensembles, as long as
 $w$ the weight is positive and continuous over $[0,1].$
\end{abstract}

\noindent
\section{Introduction}

In random matrix theory, Hankel determinants play a significant role, {\rm e.g.}
the determinants represent the partition functions, moment generating function of linear statistics, or
the distribution of the smallest or largest eigenvalue.
Chen and his collaborators have studied Hankel determinants from the point of view of polynomials orthogonal with respect to unconventional weights, typically involving a family of deformations of a classical weight.  In this paper we consider
\begin{equation}
w(x;\lambda, \alpha,\beta )=x^{\alpha}(1-x)^{\beta}\,{\rm e}^{-\lambda x},\;x\in[0,1],\;\alpha> -1,\;\beta> -1.
\end{equation}
  Here $w(x;0,\alpha,\beta)$ is the standard Jacobi weight on the interval $[0,1]$, and the factor ${\rm e}^{-\lambda x}$
 deforms $w(x; 0,\alpha,\beta)$ to $w(x; \lambda,\alpha,\beta )$.
The $n\times n$ Hankel determinants for $w(x; \lambda,\alpha,\beta )$ satisfy Painlev\'{e} \uppercase\expandafter{} transcendental differential equations in $\lambda$, and recurrence relations in $n$.
There is an extensive literature on the appearance of Painlev\'{e} equations in the unitary ensembles, see for example, \cite{BasorChen2009,ChenZhang2010,ChenMcKay2012,ChenFeigin2006,Mehta2006,TracyWidom1994} and the references therein. The current paper provides a direct and computationally effective approach to the problem, leading to some explicit results.

Generally, let $w(x)$ be a positive weight function on the interval $[0,1]$ and $\mu_{j}:=\int_0^1x^jw(x)\:dx$ be the moments for $j=0, 1, \dots $. We use the handy notation
\[\Delta_n(\vec{x}):=\prod_{1\leq j<k\leq n}(x_k-x_j)=\det \left[x_j^{k-1}\right]_{j,k=1}^{n}\]\\
for the Vandermonde determinant and introduce, as in
 the Andreief--Heine identity, the Hankel determinant
\begin{equation}\label{09}
D_n[w]:=\frac{1}{n!}\int_{[0,1]^n}\Delta_n(\vec{x})^2\prod
_{l=1}^nw(x_l)dx_l=\det\left(\mu_{j+k}\right)_{j,k=0}^{n-1}.
\end{equation}
Hankel determinants play an important role in the study of orthogonal polynomials \cite{Szego1939}, and random matrices.
The joint probability density  function of the Hermitian matrix ensemble for weight $w$ is given by (see \cite{Mehta2006,Weyl1946})
\begin{equation}p(x_{1}, \dots,x_n)=\frac{\Delta_n(\vec{x})^2}{n!D_n[w]}\prod
_{l=1}^nw(x_l),\end{equation}
where $ \{x_l:l=1,\ldots,n\}$ are the real eigenvalues of the complex Hermitian matrices $X$, and the probability measure is invariant under the unitary conjugation $X\mapsto UXU^\dagger$ for unitary $U$ and Hermitian $X$. The linear statistic $Q$ associated with a continuous real function $f$ is the random variable $Q=\sum_{j=1}^n f(x_j)$, where the variables $\{x_j:j=1,2,\dots ,n\}$ are random subject
to the unitary ensemble for the weight $w(x;\lambda,\alpha,\beta)$. In this paper, the large $n$ behavior of the Hankel determinant is obtained from a linear statistics theorem. This follows the approach of \cite{BasorChen2005,BasorChen2001,ChenLawrence1998,MCChen,MCChen2016}.

Suppose $Q$ has a density function denoted by $\mathbb{P}_{f}(Q)$, writing $\delta$ for the Dirac point mass at $0$, we determine $\mathbb{P}_f(Q)$ by the standard formula
\begin{equation}\label{ZZZ}
\mathbb{P}_f(Q)= \int_{[0,1]^n}p(x_1,\ldots,x_n)\delta\left(Q-\sum_{j=1}^nf(x_j)\right)dx_1\ldots dx_n.
\end{equation}

Suppose that $f(x)\geq 0$ for all $x\in[0,1].$ Then the moment generating function of $Q$ is denoted by $\mathcal{M}_f(\lambda,\alpha,\beta,n)$
so  $\mathcal{M}_f(\lambda,\alpha,\beta,n)$ is the Laplace transform of $\mathbb{P}_f(Q)$, in the transform variable $\lambda$. We can express the expectation of ${\rm e}^{-\lambda Q}$ by replacing $w(x)$ by $w(x){\rm e}^{-\lambda f(x)}$ as in
\begin{equation}
\mathcal{M}_f(\lambda,\alpha,\beta,n):=\int_0^\infty \mathbb{P}_f(Q){\rm e}^{-\lambda Q}dQ=\frac{\prod_{j=0}^{n-1}h_j(\lambda)}{\prod_{j=0}^{n-1}h_j(0)},\end{equation}
where
$$h_j(\lambda ):=\int_0^1P_j^2(x)w(x){\rm e}^{-\lambda f(x)}dx ,\; j\in\{0,1,\cdots\},$$
is the square of the $L^2$ norm of the polynomials $\{P_j\}_{j=0}^{\infty}$ orthogonal with respect to $w(x){\rm e}^{-\lambda f(x)}$.

In particular we take $f(x)=x$ to obtain the linear statistic $c=x_1+x_2+\cdots+x_n$, so $c$ is the center of mass of the unitary ensemble for weight $w(x;\lambda,\alpha,\beta)$.
The Hankel determinant generated by $w(x;\lambda,\alpha,\beta)$ which is denoted by
\begin{equation}\label{zhu0}
D_n(\lambda,\alpha,\beta)=\det\left(\mu_{j+k}(\lambda,\alpha,\beta)\right)_{j,k=0}^{n-1}=\det\left(\int_0^1 x^{j+k}x^\alpha(1-x)^\beta {\rm e}^{- \lambda x}dx\right)_{j,k=0}^{n-1}.
\end{equation}

Let $\left\{P_j(x)\right\}_{j=0}^\infty$ be the sequence of monic orthogonal polynomials  with respect to the weight $w(x; \lambda,\alpha,\beta )$, (over [0,1]), where $P_j(x)$ has degree $j$. An immediate consequence of orthogonality is that the polynomials satisfy a three-term recurrence relation, that is, a linear second-order difference equation, involving $P_{n+1}(x),\; P_n(x)~{\rm and}~ P_{n-1}(x)$. The $x$-independent recurrence coefficients, denoted as $\alpha_n(\lambda)$  and $\beta_n(\lambda)$, play an important role in computing the Hankel determinant $D_n(\lambda , \alpha, \beta )$ and ultimately $\mathbb{P}(c,\alpha,\beta,n)$.

This paper is organized as follows. In section two, we derive the Toda molecule equations for $\{P_j\}_{j=0}^{\infty}$ via the three-term recurrence relation for the monic polynomials orthogonal with respect to $w(x;\lambda,\alpha,\beta)=x^{\alpha}(1-x)^{\beta}{\rm e}^{-\lambda x}$, which is a semi-classical weight. We also introduce the ladder operators which raise and lower terms in sequence $\{P_j\}_{j=0}^{\infty}$. The ladder operators involve rational functions
$A_n(z)$ and $B_n(z)$ that have residues $R_n(\lambda)$ and $r_n(\lambda)$, and their properties are the main theme of this paper. We derive a pair of coupled Riccati equations and a pair of first-order difference equation for them; see Theorems 2.4-2.6. While these formulas are rather complicated, we obtain explicit solutions for the special case $\alpha=\beta=1/2$ in terms of Bessel functions of the first kind. These results are consistent with those of Basor, Chen and Ehrhardt \cite{BasorChenEhrhardt2010}, who considered $(1-x)^{\alpha}(1+x)^{\beta}{\rm e}^{-tx}$ on $x\in[-1,1]$. For general $\alpha,\beta>-1$, we do not expect closed form solutions in terms of standard transcendental functions.\par

\indent The ladder operators provide an effective and direct approach towards the Painlev\'e transcendental differential equations. In section 3, we show that, with suitable change of variable, $R_n(\lambda)\rightarrow Y_n(-\lambda)$ satisfies a particular Painlev\'{e} \uppercase\expandafter{\romannumeral 5} with specific initial conditions. Also, $r_n(\lambda)$ satisfies a Chazy II
differential equation. Let $p(n,\lambda)$ be the coefficient of the sub-leading term of our monic polynomials, then $\sigma_n(\lambda)=n\lambda+\lambda\:p(n,-\lambda)-n(n+\beta)$ satisfies the Jimbo--Miwa--Okamoto $\sigma-$form of this Painlev\'{e} \uppercase\expandafter{\romannumeral 5}. These results are of interest in their own right, and are the foundation of the asymptotic analysis in the subsequent sections.

In section 4, we compare the Hankel determinant $D_n(\lambda,\alpha,\beta)$ for the weight
$x^{\alpha}(1-x)^{\beta}{\rm e}^{-\lambda\;x},$ with the Hankel determinant for $D_n(0,\alpha,\beta)$ for the classical Jacobi weight $x^{\alpha}(1-x)^{\beta}$
when $n$ is large. With $f(x)=x$, the ratio $\mathcal{M}_{f}(\lambda,\alpha,\beta,n)=D_n(\lambda,\alpha,\beta)/D_n(0,\alpha,\beta)$
 is the moment generating of the linear statistics $\sum_{j=1}^{n}x_j$. We approximate $\mathcal{M}_f(x,\alpha,\beta,n)$
  for large but finite $n$ by the Dyson's Coulomb fluid approach and then use the Painlev\'e analysis of section 3 to compute the cumulants $\mathbb{P}(c,\alpha,\beta,n)$.
  Our method leads to asymptotic expansions with explicit and computable coefficients.
  In section 5, we replace the weight $x^{\alpha}(1-x)^{\beta}{\rm e}^{-\lambda\;x}$ by
   the the complex function $x^{\alpha}(1-x)^{\beta}{\rm e}^{{\rm i}\lambda\;x}$;
   several of the basic formulas remain valid. Thus we compute the Fourier transform of $D_n(-{\rm i}\lambda,\alpha,\beta)$,
    and hence obtain the probability density function of $c$, $\mathbb{P}(c,\alpha,\beta,n)$. Finally we study the characteristics of asymptotic expressions $\mathbb{P}(c,\alpha,\beta,n)$ denote by $\widehat{\mathbb{P}}(c,\alpha,\beta,n)$.

\section{Toda Evolution and Riccati equations}
Our first purpose in this section is to deduce two coupled Toda type equations. The general Toda hierarchy can be found, in \cite{Adler,Toda1993,Moser1975,Witten1991}. The three-term recurrence relation is an immediate consequence of the orthogonality of of $P_n(z)$, namely,
\begin{equation}
zP_n(z)=P_{n+1}(z)+\alpha_n(\lambda)P_n(z)+\beta_n(\lambda)P_{n-1}(z),
\end{equation}
with the initial conditions
\begin{equation}
P_0(z):=1~~~{\rm and}~~~\beta_0(\lambda)P_{-1}(z):=0.
\end{equation}
Here, $P_j(z)$ depends on $\lambda,\alpha,\beta$ but to simplify notation we do not always display them. Then we write our monic polynomials as,
\begin{equation*}
P_n(z,\lambda)=z^n+p(n,\lambda)z^{n-1}+\cdots+P_0(z,\lambda),
\end{equation*}
with the conditions
\begin{equation*}
P_0(z,\lambda):=1 ~~~{\rm and}~~~  p(0,\lambda):=0.
\end{equation*}
An easy consequence of the recurrence relation is
\begin{align}
\alpha_n(\lambda)&=p(n,\lambda)-p(n+1,\lambda),\label{ReC1}\\
\beta_n(\lambda)&=\frac{h_n(\lambda)}{h_{n-1}(\lambda)}=\frac{D_{n+1}(\lambda)D_{n-1}(\lambda)}{D_n^2(\lambda)}.
\label{ReC2}
\end{align}
From \eqref{ReC1} together with $p(0,\lambda)=0$, we have
\begin{equation}\label{1p1}
\sum_{j=0}^{n-1}\alpha_j(\lambda)=-p(n,\lambda).
\end{equation}
Then after some simple computation we obtain,
\begin{equation}\label{D1Dn}
\frac{d}{d\lambda}\log D_n(\lambda ,\alpha,\beta)=p(n,\lambda
),
\end{equation}
\begin{equation}\label{D1p1}
\frac{d}{d\lambda}p(n,\lambda)=\beta_n(\lambda).
\end{equation}
\\
\begin{proposition}\label{thm1}
The recursion coefficients $\alpha_n(\lambda)$ and $\beta_n(\lambda)$ satisfy the coupled Toda equations
\begin{align}
\beta'_n&=\beta_n(\alpha_{n-1}-\alpha_n) \label{Toda1},\\
\alpha'_n&=\beta_n-\beta_{n+1} \label{Toda2},
\end{align}
and the Toda molecule equation, see \cite{sogo1993},
\begin{equation}\label{Todam}
\frac{d^2}{d\lambda^2}\log D_n(\lambda ,\alpha,\beta)=\frac{D_{n+1}(\lambda ,\alpha,\beta)D_{n-1}(\lambda ,\alpha,\beta)}{D_n^2(\lambda ,\alpha,\beta)}.
\end{equation}
\end{proposition}
In what follows, we will obtain two coupled Riccati equations based on ladder operators.
The ladder operators, also called lowering and raising operators, have been applied by many authors; see for example,
\cite{BasorChen2009,Bauldry1990,Bonan1990,ChenIsmail2005,ChenIsmail1997}. In our case,  they read
\begin{align}
&\left(\frac{d}{dz}+B_n(z)\right)P_n(z)=\beta_nA_n(z)P_{n-1}(z) \label{lr1},\\
&\left(\frac{d}{dz}-B_n(z)-v^\prime(z)\right)P_n(z)=-A_{n-1}(z)P_n(z) \label{lr2},
\end{align}
where
\begin{align}
&A_n(z):=\frac{1}{h_n}\int_0^1\frac{v^\prime(z)-v^\prime(y)}{z-y}P^2_n(y)w(y)dy,\\
&B_n(z):=\frac{1}{h_{n-1}}\int_0^1\frac{v^\prime(z)-v^\prime(y)}{z-y}P_n(y)P_{n-1}(y)w(y)dy.
\end{align}
Here $w(x)={\rm e}^{-v(x)}$ and we assumed the $w(0)=w(1)=0$.

Then we obtain two fundamental supplementary conditions $(S_1),$ $(S_2)$  and a ``sum-rule"  $(S_2')$, valid for all $z$,
$$B_{n+1}(z)+B_n(z)=(z-\alpha_n)A_n(z)-v^\prime(z) \eqno{(S_1)}$$
$$1+(z-\alpha_n)(B_{n+1}(z)-B_n(z))=\beta _{n+1}A_{n+1}(z)-\beta_nA_{n-1}(z) \eqno{(S_2)}$$
$$B_n^2(z)+v^\prime(z)B_n(z)+\sum_{j=0}^{n-1}A_j(z)=\beta_nA_n(z)A_{n-1}(z).\eqno{(S_2^\prime)}$$
supplemented by the `initial' conditions,
\begin{equation*}
B_0(z)=A_{-1}(z)=0.
\end{equation*}
The equations of ($S'_2$) will be highly useful in what follows. Equations ($S_1$), ($S_2$) and ($S'_2$) can also be found in
\cite{ChenIts2010,ChenIsmail2005,ChenIsmail1997,Magnus,TracyWidom1994}.
In our problem, the linear statistic for $f(x)=x$, and the correspondingly deformed weight becomes
$$w(x;\alpha,\beta,\lambda)=x^\alpha(1-x)^\beta {\rm e}^{- \lambda x},~~x\in[0,1],~\alpha>-1,\beta>-1.$$
\begin{proposition}\label{pro2.2}
The coefficients $A_n(z)$ and $B_n(z)$ appearing in the ladder operators (obtained via integration by parts) are
\begin{align}
A_n(z)&=\frac{R_n(\lambda)}{z}+\frac{\lambda-R_n(\lambda)}{z-1}, \label{An}\\
B_n(z)&=\frac{r_n(\lambda)}{z}-\frac{n+r_n(\lambda)}{z-1},\label{Bn}
\end{align}
where
\begin{align*}
R_n(\lambda)&:=\frac{\alpha}{h_n}\int_0^1\frac{P_n^2(y)}{y}y^\alpha(1-y)^\beta {\rm e}^{-\lambda y}dy,\\
r_n(\lambda)&:=\frac{\alpha}{h_{n-1}}\int_0^1\frac{P_n(y)P_{n-1}(y)}{y}y^\alpha(1-y)^\beta{\rm e}^{-\lambda y}dy.
\end{align*}
\end{proposition}
\begin{proof}
See \cite{BasorChenEhrhardt2010}.
\end{proof}
Ultimately, the recurrence coefficients may be expressed in terms $r_n(\lambda)$ and $R_n(\lambda).$
\\
To begin with, substituting \eqref{An} and \eqref{Bn} into  $(S_1)$ and $(S'_2)$, we obtain
\begin{align}
&r_n^2-\alpha r_n=\beta_nR_nR_{n-1},\label{1S'2}\\
&(r_n+n)^2+\beta(r_n+n)=\beta_n(R_n-\lambda)(R_{n-1}-\lambda), \label{2S'2}
\end{align}
\begin{equation}
2r_n(r_n+n)-\alpha r_n+\beta r_n-\alpha n+\lambda r_n+\sum_{j=0}^{n-1}R_j=\beta_n\left[R_n\left(R_{n-1}-\lambda\right)+R_{n-1}\left(R_n-\lambda\right)\right]. \label{3S'2}
\end{equation}
After easy computations, we have,
\begin{proposition} The recurrence coefficients $\alpha_n$, $\beta_n$  are expressed in terms of $R_{n}$, $r_{n}$ as,
\begin{align}
\lambda\alpha_n&=2n+1+\alpha+\beta+\lambda-R_n, \label{C1}\\
\beta_n(\lambda^2-\lambda R_n)&=n(n+\beta)+(2n+\alpha+\beta)r_n+\frac{\lambda}{R_n}\left(r_n^2-\alpha r_n\right).
\label{C2}
\end{align}
\end{proposition}
\begin{theorem}\label{thm1}
The auxiliary variables $r_n(\lambda)$ and $R_n(\lambda)$ satisfy coupled Riccati equations
\begin{align}
\lambda R'_n&=-\alpha\lambda +R_n(2n+1+\alpha+\beta+\lambda)-R_n^2+2\lambda r_n~,  \label{DR}\\
r'_n&=\frac{R_n}{\lambda R_n-\lambda^2}\left[n(n+\beta)+(2n+\alpha+\beta)r_n+\frac{\lambda}{R_n}(r_n^2-\alpha r_n)\right]+\frac{r_n^2-\alpha r_n}{R_n}.\label{020}
\end{align}
\end{theorem}
\begin{theorem}\label{th2.5}
The auxiliary variables $r_n(\lambda)$ and $R_n(\lambda),$ satisfy non-linear second order ordinary differential equations
\begin{align}
R''_n=&\frac{1}{2\lambda^2(R_n-\lambda)R_n}\big\{(2R_n-\lambda )(\lambda R'_n)^2-2\lambda R_n^2R'_n+2R_n^5-2\alpha^2\lambda^2R_n+\alpha^2\lambda^3 \notag \\
&-[2(2n+1+\alpha+\beta)+5\lambda]R_n^4+4\lambda(2n+1+\alpha+\beta+\lambda)R_n^3\notag\\
&-[\lambda^3-\lambda(1+\alpha^2-\beta^2)+2 \lambda^{2}(2n+1+\alpha+\beta)]R_n^2\big\}, \label{D2R}\\
\big[\lambda^2r''_n&+8 r_n^3+6(2n-\alpha+\beta)r_n^2+4(n^2-2n\alpha+n\beta-\alpha\beta)r_n-2n(n+\beta)\alpha+\lambda
r'_n\big]^{2}\notag\\
& =\left(4r_n+\lambda+2n-\alpha+\beta\right)^{2}\left[4r_n(r_n-\alpha)(r_n+n)(r_n+n+\beta)+(\lambda r'_n)^2\right].\label{D2r}
\end{align}
\end{theorem}

In addition to the coupled Riccati equation, $r_n(\lambda)$ and $R_n(\lambda)$ also satisfied a pair of coupled nonlinear first order difference equations.
\begin{theorem}\label{th2.6}The auxiliary quantities $r_n(\lambda)$ and $R_n(\lambda)$ satisfy the coupled difference equations
\begin{align}
  &\lambda(r_{n+1}+r_n)=R_n^2-R_n(2n+1+\alpha+\beta+\lambda)+\lambda\alpha,\\
 &n(n+\beta)+(2n+\alpha+\beta)r_n=(r_n^2-\alpha r_n)\left(\frac{\lambda^2}{R_nR_{n-1}}-\frac{\lambda}{R_n}-\frac{\lambda}{R_{n-1}}\right),
\end{align}
for $n=0,1, \dots$ with the `initial' conditions
\begin{equation}
 r_0(\lambda)=0,~~~~~~R_0(\lambda)=\frac{(\alpha+\beta+1)M(\alpha;1+\alpha+\beta;-\lambda)}{M(1+\alpha;2+\alpha+\beta;-\lambda)},\\
\end{equation}
where $M(a;b;z)$ is the Kummer function.
\end{theorem}
From Proposition \ref{pro2.2} and Theorem \ref{th2.6}, one could, in principle, obtain the $R_n(\lambda)$ and $r_n(\lambda)$, iteratively, step by step in $n.$ 

To check that the integral representation for $r_n$ given by Proposition \ref{pro2.2} makes sense, note that,
\begin{equation}\label{1r1}
r_1(\lambda)=\alpha -\frac{(\alpha +1)(\alpha +\beta+2)  \, M(\alpha ;\alpha +\beta +1;-\lambda ) \, M(\alpha +2;\alpha +\beta +3;-\lambda )}{(\alpha +\beta+1)\, M(\alpha +1;\alpha +\beta +2;-\lambda )^2},
\end{equation}
Substitute $n=0$ into (2.26); from the fact that $r_0=0,$ and $R_0$ given (2.28), we obtain,
\begin{align}\label{2r1}
r_1(\lambda)=&\frac{(\alpha +\beta +1)^2 \, M(\alpha ;\alpha +\beta +1;-\lambda )^2}{\lambda\, M(\alpha +1;\alpha +\beta +2;-\lambda )^2}-\frac{(\alpha +\beta +1)^2 \, M(\alpha ;\alpha +\beta +1;-\lambda )}{\lambda\, M(\alpha +1;\alpha +\beta +2;-\lambda )}+\alpha \notag\\
&-\frac{(\alpha +\beta +1) \, M(\alpha ;\alpha +\beta +1;-\lambda )}{\, M(\alpha +1;\alpha +\beta +2;-\lambda )},
\end{align}
a rather large expression. However, \eqref{1r1}$-$\eqref{2r1} gives $0$.
\begin{rem}
A direct computation shows that $R_0(\lambda),$ satisfies (2.24) evaluated at $n=0$.
Also a direct computation shows that $r_1(\lambda)$ given by (2.29) satisfies (2.25) evaluated at $n=1.$
\end{rem}
\begin{rem}
In particular, if we take $\alpha=\beta=\frac{1}{2}$, then $R_n(\lambda)$ and $r_n(\lambda)$ can be represented by a Bessel function of the first kind with imaginary argument, e.g.
\begin{align*}
r_0(\lambda)&=0,~~~~~R_0(\lambda)=\frac{1}{4} \lambda  \left[\frac{I_0\left(\frac{\lambda }{2}\right)}{I_1\left(\frac{\lambda }{2}\right)}+1\right],
~~~~~r_1(\lambda)=\frac{1}{4} \left[\frac{\lambda  I_0\left(\frac{\lambda }{2}\right) I_2\left(\frac{\lambda }{2}\right)}{I_1\left(\frac{\lambda }{2}\right)^{2}}-\lambda -2\right],\\
R_1(\lambda)&=\frac{\lambda  \left[(\lambda +4) I_1\left(\frac{\lambda }{2}\right)-\lambda  I_0\left(\frac{\lambda }{2}\right)\right] \left[-\lambda  I_0\left(\frac{\lambda }{2}\right){}^2+4 I_1\left(\frac{\lambda }{2}\right) I_0\left(\frac{\lambda }{2}\right)+(\lambda +2) I_1\left(\frac{\lambda }{2}\right){}^2\right]}{2 I_1\left(\frac{\lambda }{2}\right) \left[-\lambda ^2 I_0\left(\frac{\lambda }{2}\right){}^2+\left(\lambda ^2+8\right) I_1\left(\frac{\lambda }{2}\right){}^2+2 \lambda  I_1\left(\frac{\lambda }{2}\right) I_0\left(\frac{\lambda }{2}\right)\right]}.\\
\end{align*}
This result is consistent with the corresponding case in \cite{BasorChenEhrhardt2010}. One can verify the differential equation of Theorem \ref{th2.6} for $R_0(\lambda )$  by hand calculation or Mathematica.
\end{rem}
\begin{rem}
Disregarding the integral representation of $R_n(\lambda)$ and $r_n(\lambda)$, and putting $\alpha=-k$ $ (k=0, 1, 2,\cdots)$ and $\beta=a-\alpha$ $( a\in\mathbb{R})$, we see that $R_0(\lambda)$ and $r_0(\lambda)$ are given by Laguerre polynomials,
\begin{align} \label{Pv001}
R_{0}(\lambda)=\frac{(a+1)M(-k;a+1;-\lambda)}{M(-(k-1);a+2;-\lambda)}=\frac{kL_{k}^{(a)}(-\lambda)}{L_{k-1}^{(a+1)}(-\lambda)},
\end{align}
thus
\begin{align} \label{Pv0001}
r_{1}(\lambda)=\frac{R_{0}^{2}(\lambda)}{\lambda}-\left(1-\frac{a+1}{\lambda}\right)R_{0}(\lambda)-k.
\end{align}
Thus we generate rational solutions in terms of the Laguerre polynomials. On page 21 of Appendix A, Masuda, Ohta and Kajiwara \cite{Masuda} produced such rational solutions of Painlev\'e V.
\end{rem}
\section{Painlev\'{e} \uppercase\expandafter{\romannumeral 5}, Chazy Equation and discrete $\sigma$-form }
\subsection{Painlev\'{e} \uppercase\expandafter{\romannumeral 5}}
The auxiliary quantities $R_n(\lambda)$ and $r_n(\lambda)$ maybe recast into familiar form.
We make a change of variables
$$ R_n(\lambda):=-\frac{\lambda}{Y_n(-\lambda)-1}\Longleftrightarrow\;\;Y_n(-\lambda)=1-\frac{\lambda}{R_n(\lambda)}.$$
\begin{theorem}
The quantity $Y_n(\lambda)$ satisfies the Painlev\'{e} \uppercase\expandafter{\romannumeral 5} equation
$$P_{{\rm \uppercase\expandafter{\romannumeral 5}}}\bigg(\frac{\alpha^2}{2},-\frac{\beta^2}{2},2n+1+\alpha+\beta,\frac{1}{2}\bigg),$$
namely
\begin{align} \label{Pv1}
Y_n^{\prime\prime}=\frac{3Y_n-1}{2Y_n(Y_n-1)}{Y_n^{\prime}}^2-\frac{Y_n^\prime}{\lambda}&+\frac{(Y_n-1)^2}{\lambda^2}\left(\frac{\alpha^2}{2}Y_n-\frac{\beta^2/2}{Y_n}\right)
\notag\\&+(2n+1+\alpha+\beta)\frac{Y_n}{\lambda}-\frac{1}{2}\frac{Y_n(Y_n+1)}{Y_n-1},
\end{align}
 with initial conditions
\[Y_n(0)=1~~~{\rm and}~~~ Y_n'(0)=\frac{1}{2n+1+\alpha+\beta}.\]
\end{theorem}
\begin{proof}
See also Basor, Chen and Ehrhardt \cite{BasorChenEhrhardt2010}.
\end{proof}
\begin{theorem}
The quantity $\widetilde{\sigma}_{n}$ satisfies the following Jimbo-Miwa-Okamoto \cite{JimboMiwa1981,Okamoto} $\sigma-$form of Painlev\'{e}
\textrm{V}
 \begin{equation}\label{2Pv0}
\begin{split}
\left(\lambda\widetilde{\sigma}_{n}''\right)^2&=\left[\widetilde{\sigma}_{n}-\lambda
\widetilde{\sigma}_{n}'+2\left(\widetilde{\sigma}_{n}'\right)^{2}-\left(2n-\alpha+\beta\right)\widetilde{\sigma}_{n}'\right]^2\\
&-4\widetilde{\sigma}_{n}'\left(\widetilde{\sigma}_{n}'-\alpha\right)\left(\widetilde{\sigma}_{n}'+n\right)\left(\widetilde{\sigma}_{n}'+n+\beta\right),
\end{split}
\end{equation}
with the initial conditions
\[\widetilde{\sigma}_{n}(0)= n(n+\beta) ~~~{\rm and}~~~
\widetilde{\sigma}_{n}'(0)=-\frac{n(n+\beta)}{2n+\alpha+\beta}.\]
Comparing with Jimbo-Miwa $\sigma-$form \cite{JimboMiwa1981}, (C.45), it shows
\begin{equation*}
v_{0}=0,~~~v_{1}=-\alpha,~~~v_{2}=n,~~~v_{3}=n+\beta.
\end{equation*}
\end{theorem}
\begin{proof}
For this problem, introduce,
\begin{equation}\label{sgm}
\sigma_{n}(\lambda):=n\lambda+\lambda p(n,-\lambda)- n(n+\beta).
\end{equation}
It can be shown, following \cite{BasorChenEhrhardt2010}, that,
\begin{equation}\label{2Pv}
\left(\lambda\sigma_{n}''\right)^2=\left[\sigma_{n}-\lambda
\sigma_{n}'+(2n+\alpha+\beta)\sigma_{n}'\right]^2+4\left[(\sigma_{n}')^2+\alpha \sigma_{n}'\right]\left[ \lambda
\sigma_{n}'-\sigma_{n}-n(n+\beta)\right],
\end{equation}
Let $\widetilde{\sigma}_{n}(\lambda):=-\sigma_{n}(\lambda)$ 
 we arrive at (3.2), the $\sigma-$form of Painlev\'{e} \textrm{V}.
\end{proof}
\subsection{Chazy Equation }
We will obtain an ODE satisfied by $r_n(\lambda)$ from the $\sigma$-form of Painlev\'{e} \textrm{V}. Following  \cite{LvChen}, let
\begin{align*}
 \Xi(\lambda):=&\lambda\frac{d}{d \lambda}\log D_n(-\lambda,\alpha,\beta)-n\lambda+n(n+\beta), \\
 &\Xi'(\lambda)=r_n(-\lambda).
\end{align*}
\begin{proposition} The
$r_n(\lambda)$ satisfies the following Chazy II system,
\begin{equation}\label{Chazy}
  \left(\frac{d^2\vartheta}{dz^2}-2\vartheta^3-\alpha_1\vartheta-\beta_1\right)^2=-4(\vartheta-{\rm e}^z)^2\left[\left(\frac{d\vartheta}{dz}\right)^2-\vartheta^4-\alpha_1\vartheta^2-2\beta_1\vartheta-\gamma_1\right],
\end{equation}
where
\begin{align*}
  \vartheta(z) &=2{\rm i} r_n(2{\rm i}\/{\rm e}^z)-\frac{{\rm i}\/}{2}(2n-\alpha+\beta), \\
  \alpha_1 &=\frac{1}{2}(4n^2+4n\alpha+3\alpha^2+4n\beta+2\alpha\beta+3\beta^3),\\
   \beta_1 &=-\frac{{\rm i}\/}{2}(2n+\alpha+\beta)(\alpha+\beta)(\alpha-\beta),\\
    \gamma_1 &=\frac{1}{16}(2n+\alpha-\beta)(2n-\alpha+\beta)(2n+3\alpha+2\beta)(2n+\alpha+3\beta).
\end{align*}
\end{proposition}
\subsection{The Discrete $\sigma-$form}
\begin{theorem} The quantities $\sigma_{n+1}$, $\sigma_n$ and $\sigma_{n-1}$ satisfy
\begin{align*}
 &(\sigma_n-\sigma_{n-1}-\alpha)[(2 n + \alpha + \beta)( \sigma_n+n^2 + n\beta)-n \lambda(n + \beta)] ( \sigma_n - \sigma_{n+1}+\alpha )\cdot \\
 &(2 n - \alpha + \beta - \lambda - \sigma_{n-1} + \sigma_{n+1}) + [2 n \alpha (n + \beta) + (2 \alpha + \
\lambda)\sigma_{n}+ (n^2 + n\beta + \sigma_{n})\\
& \cdot(\sigma_{n-1}- \sigma_{n+1})] [2 n \alpha (n + \beta) + (2 \alpha + \
\lambda) \sigma_{n} + (n^{2} + n\beta + \sigma_{n}) (\sigma_{n-1} - \sigma_{n+1}) \\&- \alpha \lambda (2 n - \alpha + \beta - \lambda - \sigma_{n-1} +\sigma_{n+1})] = 0,
\end{align*}
which we call the discrete $\sigma-$form.
\end{theorem}
\begin{proof}
From \eqref{3S'2}, \eqref{C1} and \eqref{sgm}, we have
\begin{equation}\label{C2r1p1}
\lambda^2\beta_n+\lambda r_n=\sigma_n(-\lambda)+n(n+\beta).
\end{equation}
Together with \eqref{ReC1}, \eqref{C1} and \eqref{sgm}, we obtain,
\begin{equation}\label{Rp1}
R_n=\alpha+\sigma_n(-\lambda)-\sigma_{n+1}(-\lambda).
\end{equation}
Then sum of \eqref{Rp1} at `$n$'  and the same at `$n-1$', leaves
\begin{equation}\label{Rnn-1}
R_n+R_{n-1}=2\alpha+\sigma_{n-1}(-\lambda)-\sigma_{n+1}(-\lambda).
\end{equation}
From \eqref{1S'2}, \eqref{2S'2} and \eqref{Rnn-1} , we get
\begin{equation}\label{C2r2p1}
-\lambda\beta_n[2\alpha+\sigma_{n-1}(-\lambda)-\sigma_{n+1}(-\lambda)-\lambda]-(2n+\alpha+\beta)r_n=n(n+\beta).
\end{equation}

Eliminating $\beta_n$ and $r_n$ in \eqref{1S'2} from \eqref{C2r1p1} and \eqref{C2r2p1}, simultaneously, changing variable $\lambda$ to $-\lambda$, then the discrete $\sigma-$form will be obtained immediately.
\end{proof}
\begin{theorem}
 Our orthogonal polynomials $P_n(z)$ satisfy a linear
second-order ode, with rational coefficients in $z$, and the residues at the poles are in terms of
$Y_n(-\lambda),$ $\sigma_n(-\lambda)$ and $d\sigma_n(-\lambda)/d\lambda$.
\begin{equation}\label{kun}
P_n''(z)+R(z)P_n'(z)+Q(z)P_n(z)=0,
\end{equation}
where
\begin{align*}
R(z):=&\frac{\alpha+1}{z}+\frac{\beta+1}{z-1}-\lambda-\frac{1}{z+1/[Y_n(-\lambda)-1]},  \\ \notag \\
Q(z):=&\frac{n(\alpha+1)-\sigma_{n}(-\lambda)}{z}+\frac{n(\lambda-\alpha-1)+\sigma_{n}(-\lambda)}{z-1}\notag \\
&+\frac{1}{z+1/[Y_n(-\lambda)-1]}\left[\frac{1}{z-1}\left(n+\frac{d\sigma_{n}(-\lambda)}{d\lambda}\right)-
\frac{1}{z}\frac{d\sigma_{n}(-\lambda)}{d\lambda}\right].
\end{align*}
\end{theorem}
\begin{proof}
Eliminating $P_{n-1}(z)$ from \eqref{lr1} and  \eqref{lr2}, we obtain a second-order linear ordinary differential
equation for $P_n(z)$. If $y(z):=P_n(z)$, then $y(z)$ satisfies the differential equation
\begin{equation}
y''(z)-\left( v'(z)+\frac{A'_n(z)}{A_n(z)}\right) y'(z)+\left(
B'_n(z)-B_n(z)\frac{A'_n(z)}{A_n(z)}+\sum_{j=0}^{n-1}A_j(z)\right)y(z)=0.
\end{equation}
Substituting \eqref{An} and \eqref{Bn} into the above equation, keeping in mind the relationship of $R_n$ amd $Y_n$, with $r_n$ and $\sigma_n$, the equation (\ref{kun}) is found via some simple computations.
\end{proof}
\begin{rem}
We can also rewrite $\sigma_n(\lambda)$ in terms of $Y_{n}(\lambda)$, and this reads,
\begin{align}\label{PV-PV}
\sigma_n(\lambda) =&\frac{1}{4Y_n(4Y_n-1)^2}\big\{\beta^2-(\lambda Y'_n)^2+\alpha^2Y_n^4+[\alpha^2+(\beta-\lambda)^2-4\alpha(\beta+2n)+2\lambda(\alpha+6n)]Y_n^2\notag \\
  &+2[\alpha(2n-\alpha+\beta-\lambda)-4n\lambda]Y_n^3+2[2n(\alpha-\lambda)+\beta(\alpha-\beta+\lambda)]Y_n\big\}.
\end{align}
\end{rem}
\section{$\mathcal{M}_{f}(\lambda,\alpha,\beta,n)$ for large $n$  and finite $n$,
 Linear Statistics and the $\sigma$-form}

In this section, our objective is to approximate  the moment generating function
$\mathcal{M}_{f}(\lambda ,\alpha ,\beta ,n)$ of the linear statistic $c=x_1+\dots +x_n$, for large $n.$

\subsection{Log-concavity of the density of the center of mass}
\begin{proposition} Suppose that $\alpha, \beta >0$, and suppose that $\left\{x_j\right\}_{j=1}^{n}$ are random subject to the Jacobi unitary ensemble for the weight $w(x, 0 )$. Then the center of mass $c$ has a log-concave probability density function $\mathbb{P}(c, \alpha, \beta ,n)$.\end{proposition}

\begin{proof} We can view $c$ as the center of mass or as the trace of a Hermitian matrix.
Let $M_{n}^{h}(\mathbf{C})$ be the space of $n\times n$ complex Hermitian matrices, which we regard as a complex inner product space with the inner product $\langle X,Y\rangle={\rm trace}\left(XY^{*}\right)$. Let $v: (0,1)\rightarrow \mathbf{R}$ be convex and twice continuously differentiable, and suppose $v(x)=\infty$ for $x<0$ and $x>1$. Now  let $V(x)={\rm trace}~v(X)$ for $X\in M_{n}^{h}(\mathbf{C})$; then there exists $Z_{n}>0$ such that
\begin{equation}\label{7.3}
\mu_n(dX)=Z_{n}^{-1}{\rm e}^{-V(X)}dX
\end{equation}
defines a probability measure on $M_{n}^{h}(\mathbf{C})$ where $dX$ is Lebesgue measure on the entries that are on or above the leading diagonal. The crucial point is that the function $V:M_n^h(\mathbf{C})\rightarrow \mathbf{R}$ is convex, as we now show; compare \cite{B1}. Let $\left\{\xi_j\right\}_{j=1}^n$ be an orthonormal basis of $\mathbf{C}^n$ given by eigenvectors of $X\in M_n^h$ that correspond to eigenvalues $\left\{x_j\right\}_{j=1}^{n}$; for $X$ in a set of full Lebesgue measure, we can assume that all the $x_j\in \mathbf {R}$ are distinct. Then by the Rayleigh--Ritz formula
\begin{align} \label{4.2}\bigl\langle{\rm {Hess}}\, V, Y\otimes Y\bigr\rangle&=\sum_{j=1}^n v''(x_j)  \bigl\langle Y\xi_j, \xi_j\bigr\rangle^2_{\mathbf{C}^n}\nonumber\\ &\quad+\sum_{j,k=1;j\neq k}^{n} {\frac{v'(x_j)-v'(x_k)}{x_j-x_k}}\bigl\langle Y\xi_j  , \xi_k\bigr\rangle_{\mathbf{C}^n}\end{align}
 \noindent which is nonnegative by convexity of $v$.
The matrix $X$ has a system of coordinates given by the real and imaginary parts of entries $X=\sum_{j=1}^n {\rm e}_{jj} u_{jj}+\sum_{j<k} {\rm e}_{jk} (u_{jk}+{\rm i}v_{jk})$, where ${\rm e}_{jk}$ are the standard matrix units.
\noindent We introduce a new orthonormal basis $(Y_1, \dots ,Y_{n^2})$ for $M_n^h(\mathbf{C})$ where $Y_1=\sum_{j=1}^n {\rm e}_{jj}/\sqrt {n}$, so that the new variables are $y_k=\langle X, Y_k\rangle$ for $k=1, \dots, n^2$; in particular $y_1={\rm trace}(X)/\sqrt {n}$.  Thus we change variables to $Y=U(X)$ where $U: M_n^h (\mathbf{C})\rightarrow M_n^h(\mathbf{C})$ is a unitary linear transformation. The function $W(Y)=V(U^{-1}(Y))$ is also convex, so by Pr\'ekopa's theorem from page 106 of \cite{Boyd}, the marginal density
\begin{equation}g_1(y_1)=\int_{{\bf R}^{n^2-1}} Z^{-1} e^{-W(y_1, \dots, y_{n^2})}dy_2\dots dy_{n^2}\end{equation}
\noindent is a probability density function such that $-\log g_1(y_1)$ is convex.
 \vskip2mm
In particular, we can take $v(x)=-\alpha \log x-\beta \log (1-x)$, since $v''(x)\geq (\alpha^{1/3}+\beta^{1/3})^3$ for $0<x<1$, and
$$y_1={\rm trace}(X)/\sqrt {n}=\sum_{j=1}^nx_{j}/\sqrt{n}=c/\sqrt{n}.$$
The Vandermonde $\Delta_n (\vec{x})^2$ arises as a Jacobian factor when one passes down from $X\in M_n^h(\mathbf{C}) $ to $x\in \mathbf{R}^n$, so by rescaling we can write  the probability density function of $c$ as $\mathbb{P}(c, \alpha, \beta ,n)={\rm e}^{-v_n(c)}$, where
$v_n:(0, \infty )\rightarrow (-\infty, \infty )$ is convex.
\end{proof}

\noindent From this result, we have $\mathcal{M}_f(\lambda , \alpha , \beta ,n) =\int_0^\infty {\rm e}^{-\lambda c-v_n(c)}dc.$
By the Cauchy--Schwartz inequality,
$${{d^2}\over{d\lambda^2}}\log {\cal M}_f(\lambda ,\alpha, \beta,n )\geq 0\qquad (\lambda >0).$$
\noindent so $\log {\cal M}_f(\lambda ,\alpha, \beta ,n )$ is convex. Let
$$v_n^*(\lambda )=\sup\{ \lambda c-v_n(c): c>0\}$$ be the Legendre transform of $v_{n}$, which is also convex. From the definition, we have an optimal inequality ${\rm e}^{-\lambda c-v_n(c)}\leq {\rm e}^{v_n^*(-\lambda )}$. According to Laplace's approximation method for integrals, $v_n^*(-\lambda )$ provides a first approximation to $\log \mathcal{M}_f(\lambda , \alpha , \beta ,n)$. In the next subsection, we refine this idea by using Dyson's method for Coulomb fluids.
  \\

\vskip2mm
\subsection{Dyson's Coulomb Fluid}
\noindent In this subsection, we show that the moment generating function of linear statistics can be computed via the Dyson's Coulomb Fluid approach, as can be found
\cite{ChenLawrence1998}. We first present some background to the Linear Statistics formula, originating from the Coulomb fluid. Consider the quotient of the Hankel determinants
\begin{equation}\label{rZn}
\frac{D_n(\lambda,\alpha,\beta)}{D_n(0,\alpha,\beta)}:=\mathrm{e}^{-(F_n(\lambda)-F_n(0))},
\end{equation}
where
\begin{equation}\label{Zn0}
D_n(\lambda,\alpha,\beta):=\int_{[0,1]^n}\exp\left[-E(x_1,\dots,x_n)-\lambda\sum_{j=1}^nf(x_j)\right]\;\;dx_1\dots dx_n.
\end{equation}
Interpreting $\{x_{k}:k=1,2,\ldots,n\}$ as the positions of $n$ identically charged particles on the real line, we see that
\begin{equation*}\label{E}
E(x_1,\dots,x_n):=-2\sum_{1\leq j<k\leq n }\log\mid x_j-x_k\mid+n\sum_{j=1}^n v_{0}(x_j),
\end{equation*}
with
\begin{equation*}
  v_0(x)=-\frac{\alpha}{n}\log x-\frac{\beta}{n}\log(1-x),
\end{equation*}
is the total energy of the $n$ repelling, classical charged particles which are confined by a common external potential
$nv_{0}(x)$. The linear statistic associated with $f(x)$, acts as a perturbation to the original system, which modifies the external potential.

For large enough $n$, the collection particles can be approximated as a continuous
fluid with a certain density $\sigma (x)$ supported on a single interval $(a,b)\subseteq [0,1]$, see \cite{Dyson}. This density
corresponds to the equilibrium density of the fluid, obtained by the constrained minimization of the free-energy function, $F[\sigma]$, {\rm i.e.}
\[\min_{\sigma} F[\sigma] \quad \text{subject to}\quad \int_a^b\sigma(x)dx=1,\quad \sigma (x)\geq 0\]
with
\[F[\sigma]:=\int_a^b\sigma(x)(n^2v_0(x)+\lambda nf(x))dx-n^2\int_a^b\int_a^b\sigma(x)\log|x-y|\sigma(y)\,dxdy.\]

Upon minimization \cite{MTsuji}, the equilibrium density $\sigma(x)$ satisfies the integral equation
\begin{equation}\label{08}
v_0(x)+\frac{\lambda}{n}f(x)-2\int_a^b\log|x-y|\sigma(y)dy=A,\quad x\in[a,b],
\end{equation}
where $A$ is the Lagrange multiplier which imposes the constraint that the equilibrium density has total charge of unity, {\rm i.e.}
$\int_a^b\sigma(x)dx=1$.

We note that $A$ and $\sigma$ depend upon $\lambda$ and $n$, but not upon $x.$  The (\ref{08}) is converted into a singular
integral equation by taking a derivative with respect to $x$,
\[2{\rm PV}\int_a^b\frac{\sigma(y)}{x-y}dy=v'_0(x)+\frac{\lambda}{n}f'(x),\quad x\in[a,b],\]
where PV denotes the Cauchy principal value. The boundary condition on $\sigma(x)$ is that it vanishes at $x=a$ and $x=b$. Supposing $v_0(x)$ is convex, we can find the solution to this problem; see \cite{ChenLawrence1998}. Taking the
optimal
$\sigma(x;\lambda,n)$ in the form of
\begin{equation}\label{1sgm0}
\sigma(x)=\sigma(x;\lambda,n)=\sigma_0(x)+\frac{\tilde{\varrho}(x)}{n},
\end{equation}
where
\begin{equation}\label{1sgm0}
\sigma_0(x)=\frac{\sqrt{(b-x)(x-a)}}{2\pi^2}{\rm PV}\int_a^b\frac{v'_0(x)-v'_0(y)}{(x-y)\sqrt{(b-y)(y-a)}}dy,
\end{equation}
denotes the density $\sigma(x)$ of the original system that is with respect to the weight $w_{0}(x)=x^{\alpha}(1-x)^{\beta}$,
and
\begin{equation}\label{vr1}
\tilde{\varrho}(x)=\tilde{\varrho}(x;\lambda)=\lambda\varrho(x)=\frac{\lambda}{2\pi^2\sqrt{(b-x)(x-a)}}{\rm PV}\int_a^b\frac{\sqrt{(b-y)(y-a)}}{y-x}f'(y)dy,
\end{equation}
represents the deformation of the density due to the ``perturbation", $\lambda f(x)/n.$
\begin{theorem}
For sufficiently large $n$, the moment generation function $\mathcal{M}_f(\lambda,\alpha,\beta,n)$ has the following asymptotic expression,
\begin{align}\label{Ml}&
\mathcal{M}_f(\lambda,\alpha,\beta,n)=\frac{D_n(\lambda,\alpha,\beta)}{D_n(0,\alpha,\beta)}\cr
  &~\sim~\exp \left[\frac{a^2+2ab-b^2}{16}\lambda^2-\left(n+\frac{\alpha+\beta}{2}\right)
  \left(\sqrt{(a-1)(b-1)}-\frac{a+b}{2}+1\right)\lambda\right]
\end{align}
where $a$ and $b$ are defined in (\ref{aA}).
\end{theorem}
\begin{proof}
From above results, for sufficiently large $n$, the ratio \eqref{rZn} will be the approximated by

\begin{equation}
\frac{D_n(\lambda,\alpha,\beta)}{D_n(0,\alpha,\beta)} ~\sim~ \exp\left[{-\frac{\lambda^2 }{2}J_1-\lambda J_2}\right],
\end{equation}
where
\begin{equation}\label{S1S2}
J_1=\int_a^bf(x)\varrho(x)dx,\qquad J_2=n\int_a^bf(x)\sigma_0(x)dx.
\end{equation}
In our problem,
\begin{equation}\label{010}
  f(x) =x ,
\end{equation}
\begin{equation}\label{011}
  v_0(x) =-\frac{\alpha}{n}\log x-\frac{\beta}{n}\log(1-x)=-\frac{\log w(x)}{n}.
\end{equation}
with
\begin{equation}\label{012}
w(x)=x^{\alpha}(1-x)^{\beta},~x\in[0,1],~\alpha>-1,\beta>-1.
\end{equation}

We first consider the limiting density $\sigma_0$. In \cite{BasorChen2009}, where the limiting density denoted by $\rho(x)$ respected to the classical Jacobi weight supported on $[-1,1]$
\begin{equation}
w_{Jac}(x)=(1-x)^\alpha(1+x)^\beta,\quad x\in[-1,1],
\end{equation}
is given by
\begin{equation}\label{4.14}
\rho(y)=\frac{1}{\pi}\frac{n+(\alpha+\beta)/2}{1-y^2}\sqrt{(B_n-y)(y-A_n)}\quad y\in(A_n,B_n),
\end{equation}
with
\begin{align}
  A_n&:=\frac{1}{(2n+\alpha+\beta)^2}\left[\beta^2-\alpha^2-4\sqrt{n(n+\alpha)(n+\beta)(n+\alpha+\beta)}\right], \notag\\
  B_n&:=\frac{1}{(2n+\alpha+\beta)^2}\left[\beta^2-\alpha^2+4\sqrt{n(n+\alpha)(n+\beta)(n+\alpha+\beta)}\right].
\end{align}
To investigate the large $n$ behavior, make the replacement
\begin{equation}
\alpha\rightarrow n\alpha ~~~{\rm and}~~~\beta\rightarrow n\beta.
\end{equation}
The limit $n\rightarrow\infty$ gives $A_n\rightarrow A $ and $B_n\rightarrow B$ where
\begin{align}
  A&:=\frac{1}{(2+\alpha+\beta)^2}\left[\beta^2-\alpha^2-4\sqrt{(1+\alpha )(1+\beta )(1+\alpha +\beta)}\right],
  \notag\\
  B&:=\frac{1}{(2+\alpha+\beta)^2}\left[\beta^2-\alpha^2+4\sqrt{(1+\alpha)(1+\beta)(1+\alpha+\beta)}\right].
\end{align}
We now translate $(-1,1)$ to $(0,1)$ so
\begin{equation}\label{013}
w(x)=\frac{{w}_{Jac}(1-2x)}{2^{\alpha+\beta}},
\end{equation}
hence,
\begin{equation}\label{014}
v'_0(x)=-\frac{w'(x)}{n\cdot w(x)}=-\frac{w'_{Jac}(1-2x)}{n\cdot w_{Jac}(1-2x)}.
\end{equation}
Substituting (\ref{014}) into \eqref{1sgm0} gives the desired result
\begin{equation}\label{2sgm0}
\sigma_0(x)=\frac{2}{n}\rho(1-2x)=\left[1+\frac{\alpha+\beta}{2}\right]\frac{\sqrt{(b-x)(x-a)}}{\pi x(1-x)},
\end{equation}
where $0<a<x<b<1$, with
\begin{equation}\label{aA}
a:=\frac{1-B}{2}~~~{\rm and}~~~ b:=\frac{1-A}{2}.
\end{equation}

For $S_{2}$, substituting $f(x)=x$ into \eqref{vr1}, we have
\begin{align}
\varrho(x)
&=\frac{1}{2\pi^2\sqrt{(b-x)(x-a)}}{\rm PV}\int_a^b\frac{\sqrt{(b-y)(y-a)}}{y-x}dy\notag \\
&=\frac{(a+b)/2-x}{2\pi\sqrt{(b-x)(x-a)}}.  \label{vr2}
\end{align}

Substituting \eqref{2sgm0} and \eqref{vr2} into \eqref{S1S2} we obtain, for large $n$,
\begin{equation*}
J_1=-\frac{1}{8}\left(a^2+2ab-b^2\right),~~ J_2=\frac{2n+\alpha+\beta}{2}\left[\sqrt{(a-1)(b-1)}-\frac{a+b}{2}+1\right]\label{015},
\end{equation*}
with $a$ and $b$ given by (4.21).  Since $\log  {\cal M}_f(\lambda ,\alpha, \beta,n )$  is convex on $(0, \infty )$, one finds $J_1\leq 0$, so $a^2+2ab-b^2\geq 0$.
\end{proof}

\subsection{Cumulants of the distribution of the center of mass}

As a function of $\lambda$, our ${\cal M}_f(\lambda , \alpha, \beta ,n )$ is analytic on a neighbourhood of $\lambda =0$ with ${\cal M}_f(0 , \alpha, \beta ,n )=1;$ hence there is a convergent power series expansion
$$\log {\cal M}_f(\lambda , \alpha, \beta ,n )=\sum_{m=1}^\infty {{\kappa_m(n)}\over{m!}}(-\lambda )^m,$$
\noindent where the $\kappa_m(n)$ are the cumulants of $\mathbb{P}(c,\alpha,\beta,n)$. Combining (2.6) and (3.3), we have
$$\lambda {{d}\over{d\lambda }}\log{\cal M}_f(\lambda , \alpha, \beta ,n )=\lambda {{d}\over{d\lambda }}\log D_n(\lambda , \alpha ,\beta )=-\sigma_n(-\lambda )-n\lambda -n(n+\beta ),$$
\noindent so the Taylor coefficients of $\sigma_n(-\lambda )$ determine these cumulants. \par
\indent We can also write
$${\cal M}_f(\lambda , \alpha, \beta ,n )=\exp\left[ -{{\lambda^2}\over{2}}J_1-\lambda J_2+\int_0^\lambda {{C_n(-s)}\over{s}}ds\right],$$
\noindent where $C_n(-s)$ captures the error in the approximation (4.8) and the higher order cumulants. In the following results, we compute the power series expansion of $\sigma_n(\lambda )$, starting with the simplest case $\alpha=\beta=0.$
\begin{proposition}Suppose $\alpha =\beta=0$. Then $\sigma_n(\lambda)$ has a convergent power series in $\lambda$,
\begin{equation}\label{sgmb}
  \sigma_n(\lambda)=-n^2+n\lambda-\sum_{m=1}^\infty b_m(n)(-\lambda)^m,
\end{equation}
where the coefficients are
\begin{align}\label{bn}
  b_1(n) &=-\frac{n}{2},~~b_2(n) =\frac{n^2}{4(4 n^2-1)},~~b_4(n) =\frac{ n^2}{16 (4 n^2-9) (4 n^2-1)^2},\notag\\
b_6(n)&=\frac{n^2( 2n^2 + 1)}{ 32 (4 n^2-1)^3 ( 16 n^4- 136 n^2+225  )}, \notag\\
  b_8(n) &=\frac{n^2(64 n^6- 32 n^4- 392 n^2-45  )}{256(4 n^2-1  )^4 (4 n^2-9 )^2  (16 n^4 - 296 n^2+1225  )}.
\end{align}
\end{proposition}

\begin{proof}
We express of $\sigma_n(\lambda)$ in terms $C_n(\lambda)$, that is,
\begin{equation}\label{Sigma2}
\sigma_n(\lambda)=J_1\lambda^2-J_2\lambda+n\lambda-n(n+\beta)-C_n(\lambda).
\end{equation}
Substituting this into (\ref{2Pv}), we obtain a second order ode of $C(\lambda)$. {(The case of $\alpha\neq 0,\beta\neq 0$ will be studied later in the section.)}
Imposing the hypothesis that  $\alpha=\beta=0$, we have $a=0,~b=1$, so

\begin{align}\label{Cllll}
\lambda^2\left[C_n''-2\left(C_n''\right)^2\right]&=8\lambda\left( C_n'\right)^3-8 C_n\left( C_n'\right)^2-\left(8n^{2}+7\lambda^2\right) \left(C_n'\right)^2+\left(6n^2\lambda +\frac{3\lambda^3}{2}\right)C_n'\notag \\
&+8\lambda C_n C_n'-2C_n^2-\left(2n^2+\lambda^2\right) C_n-\frac{3\lambda^4}{32}-\left(\frac{3n^2}{4}-\frac{1}{8}\right)\lambda^2,
\end{align}
with $C_n(0)=0, \; C_n'(0)=0.$

If $\widehat{C}_{n}(\lambda):=C_{n}(-\lambda)$, a little computation show that $\widehat{C}_{n}(\lambda)$ satisfies (\ref{Cllll}). An outcome of this is that $C_{n}(\lambda)$ is even in $\lambda$.
%


Hence with the series expansion in $\lambda^{2},$  we find (for a fixed $n$) 
\begin{equation*}
  C_n(\lambda)=\sum_{j=0}^\infty a_j \lambda^{2j},~\lambda\rightarrow0,
\end{equation*}
The coefficient $a_{1}$ satisfies a quadratic equation; one of the solutions is $a_{1}=1/8$ leading to the solution $\lambda^{2}/8$ for (\ref{Cllll}). We take the other solution for $a_{1}$, giving
\begin{align}\label{Cl3}
 C_n(\lambda)=&\frac{6n^2-1}{8(4n^2-1)}\lambda^2+\frac{n^2}{16(4n^2-9)(4n^2-1)^2}\lambda^4+\frac{n^2(2n^2+1)\lambda^6}{32(4n^2-1)^3(16n^4-136n^2+225)}\notag \\
 &+\frac{ n^2 ( 64 n^6- 32 n^4- 392 n^2 -45  )}{
 256 ( 4 n^2-9)^2 ( 4 n^2-1)^4 (16 n^4- 296 n^2+1225 )}\lambda^8+\mathcal{O}(\lambda^{10}),
\end{align}
Substituting \eqref{Cl3} into \eqref{Sigma2} gives the expansion of $\sigma_n(\lambda)$ in $\lambda$.
\end{proof}

\subsection{Coefficients of $\sigma_n(-\lambda,\alpha,\beta)$ for $\alpha, \beta >-1$, $\alpha\neq0$, $\beta\neq0$}
We relax the special assumptions on $\alpha$ and $\beta$, and list the following coefficients.
\begin{align*}
b_1(n,\alpha,\beta)&=-\frac{n (\alpha +n)}{\alpha +\beta +2 n}\\
b_2(n,\alpha,\beta)&=\frac{n (\alpha +n) (\beta +n) (\alpha +\beta +n)}{(\alpha +\beta +2 n-1) (\alpha +\beta +2 n)^2 (\alpha +\beta +2 n+1)}\\
b_3(n,\alpha,\beta)&=n (\alpha -\beta ) (\alpha +\beta ) (\alpha +n) (\beta +n) (\alpha +\beta +n)/[(\alpha +\beta +2 n-2) \\
&\cdot(\alpha +\beta +2 n-1) (\alpha +\beta +2 n)^3 (\alpha +\beta +2 n+1) (\alpha +\beta +2 n+2)]\\
b_4(n,\alpha,\beta)&=\{(\alpha +\beta -1) (\alpha +\beta )^2 (\alpha +\beta +1) \left(\alpha ^2-3 \alpha  \beta +\beta ^2+1\right)-n^4 \left(8 \alpha ^2+8 \beta ^2-4\right)\\
&-8 n^3 (\alpha +\beta ) \left(2 \alpha ^2+2 \beta ^2-1\right)-2 n^2 [3 \alpha ^4+12 \alpha ^3 \beta +\alpha ^2 \left(18 \beta ^2-7\right)\\
&+6 \alpha  \beta  (2 \beta ^2-1)+3 \beta ^4-7 \beta ^2+2]+2 n (\alpha +\beta ) [\alpha ^4-4 \alpha ^3 \beta +\alpha ^2 \left(5-10 \beta ^2\right)\\
&+\alpha  \left(2 \beta -4 \beta ^3\right)+\beta ^4+5 \beta ^2-2]\}n(\alpha +n) (\beta +n) (\alpha +\beta +n)/[(\alpha +\beta +2 n)^4 \\
&\cdot(\alpha +\beta +2 n-3) (\alpha +\beta +2 n-2)(\alpha +\beta +2 n-1)^2(\alpha +\beta +2 n+1)^2 \\
&\cdot(\alpha +\beta +2 n+2) (\alpha +\beta +2 n+3)]\\
b_5(n,\alpha,\beta)&=\{ (\alpha +\beta -1) (\alpha +\beta )^2 (\alpha +\beta +1)\left(\alpha ^2-5 \alpha  \beta +\beta ^2+5\right) +16 n^6+48 n^5 (\alpha +\beta )\\
&+4 n^4 \left(7 \alpha ^2+30 \alpha  \beta +7 \beta ^2-6\right)-8 n^3 [(3 \alpha -\beta ) (\alpha -3 \beta )+6] (\alpha +\beta )\\
&-2 n^2 [\left(14 \alpha ^2-11\right) \beta ^2+18 \alpha  \left(\alpha ^2+2\right) \beta +11 \alpha ^2 \left(\alpha ^2-1\right)+18 \alpha  \beta ^3+11 \beta ^4+10]\\
&-2 n (\alpha +\beta ) [\alpha ^4+10 \alpha ^3 \beta +\alpha ^2 \left(18 \beta ^2-23\right)+2 \alpha  \beta  \left(5 \beta ^2+6\right)+\beta ^4-23 \beta ^2+10]\}\\
&\cdot n (\alpha -\beta ) (\alpha +\beta ) (\alpha +n) (\beta +n)(\alpha +\beta +n)/[(\alpha +\beta +2 n-4)(\alpha +\beta +2 n)^5 \\
&\cdot(\alpha +\beta +2 n-3) (\alpha +\beta +2 n-2) (\alpha +\beta +2 n-1)^2  (\alpha +\beta +2 n+1)^2 \\
&\cdot(\alpha +\beta +2 n+2) (\alpha +\beta +2 n+3) (\alpha +\beta +2 n+4)]
 \end{align*}

\begin{theorem} Then
$\sigma_n(\lambda,\alpha,\beta)$ has a convergent power series
\begin{equation}\label{sgmab}
  \sigma_n(\lambda,\alpha,\beta)=-n(n+\beta)+n\lambda-\sum_{m=1}^\infty b_m(n,\alpha,\beta)(-\lambda)
  ^m
\end{equation}
where the first few $b_m=b_m(n,\alpha,\beta),\;m=1,2,\dots, 5$ are listed above.
\end{theorem}
\begin{proof}
Substituting \eqref{sgmab} into \eqref{2Pv}, we obtain the results.
\end{proof}


Let $G$ be the Barnes $G$-function, defined by the functional equation, $G(z+1)=\Gamma(z)\:G(z)$.
For $n$ equal to a positive integer,
$G(1+n)=\prod_{j=1}^{n-1}j!$.

\begin{theorem}\label{zhuT1} The Hankel determinant $D_n(\lambda,\alpha,\beta)$ has the asymptotic expression
 \begin{align}\label{Dlab}
D_n(\lambda,\alpha,\beta)=&D_n(0,\alpha,\beta)\exp\left[\int_0^\lambda\frac{-\sigma_n(-s,\alpha,\beta)-ns-n(n+\beta)}{s}\, ds\right] \notag \\
=&D_n(0,\alpha,\beta)\exp\bigg[b_1(n,\alpha,\beta)\lambda+\frac{b_2(n,\alpha,\beta)}{2}\lambda^2+
\frac{b_3(n,\alpha,\beta)}{3}\lambda^3\notag\\
&~~~~~~~~~~~~~~~~~~~~~+\frac{b_4(n,\alpha,\beta)}{4}\lambda^4+\frac{b_5(n,\alpha,\beta)}{5}\lambda^5+\cdots \bigg],
\end{align}
where $D_n(0,\alpha,\beta)$ is given in (\cite{Mehta2006}, p. 310) by
\begin{align*}
D_n(0,\alpha,\beta)&=\prod_{j=0}^{n-1}\frac{\Gamma(j+2)\Gamma(j+\alpha+1)\Gamma(j+\beta+1)}{n!\Gamma(n+j+\alpha+\beta+1)}\notag\\
&=G(n+1)\frac{ G(n+\alpha +1) G(n+\beta +1) G(n+\alpha +\beta +1)}{G(\alpha +1) G(\beta +1) G(2 n+\alpha +\beta +1)}.
\end{align*}
\end{theorem}


\begin{corollary} Suppose that $\alpha=\beta=0$. Then $D_n(\lambda,0,0)$ has the following expansion,
\begin{equation}\label{Dl00}
D_n(\lambda,0,0)=\frac{G(n+1)^4}{G(2n+1)}\cdot\exp\left[-\frac{n\lambda}{2}+\frac{b_2(n)}{2}\lambda^2+\frac{ b_4(n)}{4}\lambda^4  +\frac{ b_6(n)}{6}\lambda^6+\cdots \right],
\end{equation}
where $b_m(n)$ are in agreement with those in \eqref{bn}.
\end{corollary}

\section{The asymptotic expression of $\mathbb{P}(c,\alpha,\beta,n)$}
We extend the definition of $D_n(\lambda , \alpha ,\beta )$ via the formula (\ref{zhu0}) to complex $\lambda$ and obtain an entire function. Then the Laplace inversion formula applied to (\ref{ZZZ}) gives
\begin{equation}\label{zhu}
\mathbb{P}(c,\alpha, \beta,n)={{1}\over{2\pi D_n(0, \alpha ,\beta )}}\int_{-\infty}^\infty {\rm e}^{-{\rm i} c\lambda }D_n(-{\rm i}\lambda , \alpha ,\beta )d\lambda.
\end{equation}
 In Appendix A, we give more details about the properties of the complex function $D_n(\lambda , \alpha ,\beta )$ and the support of $\mathbb{P}(c,\alpha, \beta,n)$. By Theorem \ref{zhuT1}, we have
$$\mathbb{P}(c,\alpha, \beta,n)={{1}\over{2\pi}}\int_{-\infty}^\infty \exp \left[-(b_1+c){\rm i}\lambda -{{b_2}\over{2}} \lambda^2 +{{b_3}\over{3}}{\rm i}\lambda^3 +{{b_4}\over{4}}\lambda^4+\dots \right] d\lambda ,$$
\noindent where $b_m=b_m(n, \alpha ,\beta )$ are the cumulants of  $\mathbb{P}(c,\alpha, \beta ,n)$, up to factors involving only $m$, and the values of the $b_m$ are listed in Subsection 4.4.

Edgeworth showed how to recover a probability density function from the cumulants by what is known as type A series, See \cite{SO}.

\begin{theorem}
Then $\mathbb{P}(c,\alpha,\beta,n)$ has the following asymptotic expansion, 
$$\mathbb{P}(c,\alpha,\beta,n)\sim \widehat{\mathbb{P}}(c,\alpha,\beta,n)$$
\noindent where
\begin{align}\widehat{\mathbb{P}}(c,\alpha,\beta,n)=
& \frac{1}{\sqrt{2\pi b_2}}\exp{\left[-\frac{\left(c+b_1\right){}^2}{2 b_2}\right]}
   \bigg\{1-\frac{  \left(c+b_1\right) \left[\left(c+b_1\right){}^2-3 b_2\right]b_3}{3 b_2^3}\notag\\ \notag \\ &+\frac{ \left[\left(c+b_1\right){}^4-6 b_2 \left(c+b_1\right){}^2+3 b_2^2\right]b_4}{4 b_2^4}+\cdots\bigg\} \label{Gammacc},
\end{align}
and $b_m(n,\alpha,\beta)$ are listed in subsection 4.4.
\end{theorem}

In Appendix A, we show that $\mathbb{P}(c, \alpha, \beta ,n)$ is supported on $[0,n]$. This does not conflict with the approximate expression
$\widehat{\mathbb{P}}(c, \alpha, \beta ,n)$, since the Gaussian factor $\exp \left[-(c+b_1)^2/(2b_2)\right]$ decays very rapidly outside $[0,n]$.


\begin{theorem}
Suppose $\alpha=\beta=0$. Then the probability density function of the center of mass,  $\mathbb{P}(c,0,0,n), $ has the asymptotic expression
$$ \mathbb{P}(c,0,0,n)~\sim~ \widehat{\mathbb{P}}(c,0,0,n)$$
where
\begin{align}\widehat{\mathbb{P}}(c,0,0,n)=&
 \sqrt{\frac{2(4n^2-1)}{n^2\pi}}\exp\left[2\Bigl(\frac{1}{n^2}-4\Bigr)\Bigl(c-\frac{n}{2}\Bigr)^2\right]\\
&\times\bigg[1+\frac{ \eta_{n_1}(c)}{4n^2-9}\notag +\frac{\eta_{n_2}(c)}{16n^4-136n^2+225}+\cdots \bigg]\label{Agammacc},
\end{align}
in which the coefficients are
\begin{align*}
 \eta_{n_1}(c)&= \frac{4\left(c-\frac{n}{2}\right)^4}{n^6}-\frac{2\left(c-\frac{n}{2}\right)^2\left[16\left(c-\frac{n}{2}\right)^2-3\right]}{n^4}
 +\frac{64\left(c-\frac{n}{2}\right)^4-24\left(c-\frac{n}{2}\right)^2+\frac{3}{4}}{n^2},\notag \\
  \eta_{n_2}(c)
  &= -\frac{64\left(c-\frac{n}{4}\right)^6/3}{n^{10}}+\frac{8192\left(c-\frac{n}{2}\right)^6/3-2560\left(c-\frac{n}{2}\right)^4+480\left(c-\frac{n}{2}\right)^2-10}{n^2}\notag\\
  &+\frac{80\left(c-\frac{n}{4}\right)^4\left[8\left(c-\frac{n}{4}\right)^2/3-1\right]}{n^8}-\frac{4\left(c-\frac{n}{4}\right)^2\left[128\left(c-\frac{n}{4}\right)^4
  +120\left(c-\frac{n}{2}\right)^2-15\right]}{n^6} \notag \\
  &-\frac{2048\left(c-\frac{n}{4}\right)^6/3-120\left(c-\frac{n}{2}\right)^2+5}{n^4}.
\end{align*}
\end{theorem}
\begin{figure}
\centering
\subfigure[$n$=2]{
 \centering
  \includegraphics[width=7
cm,height=3.9cm]{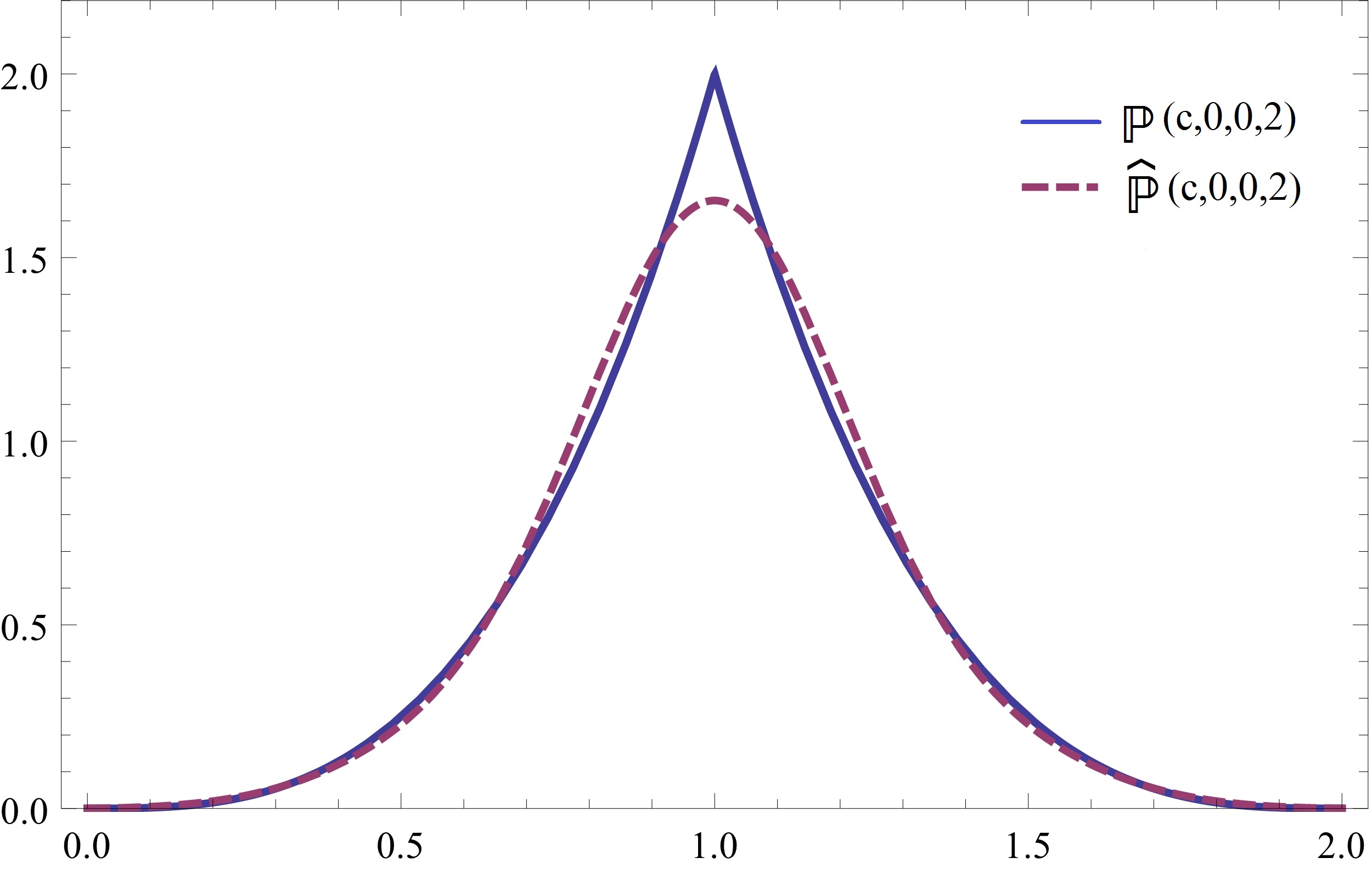}}
\subfigure[$n$=3]{
   \centering
  \includegraphics[width=7
cm,height=4cm]{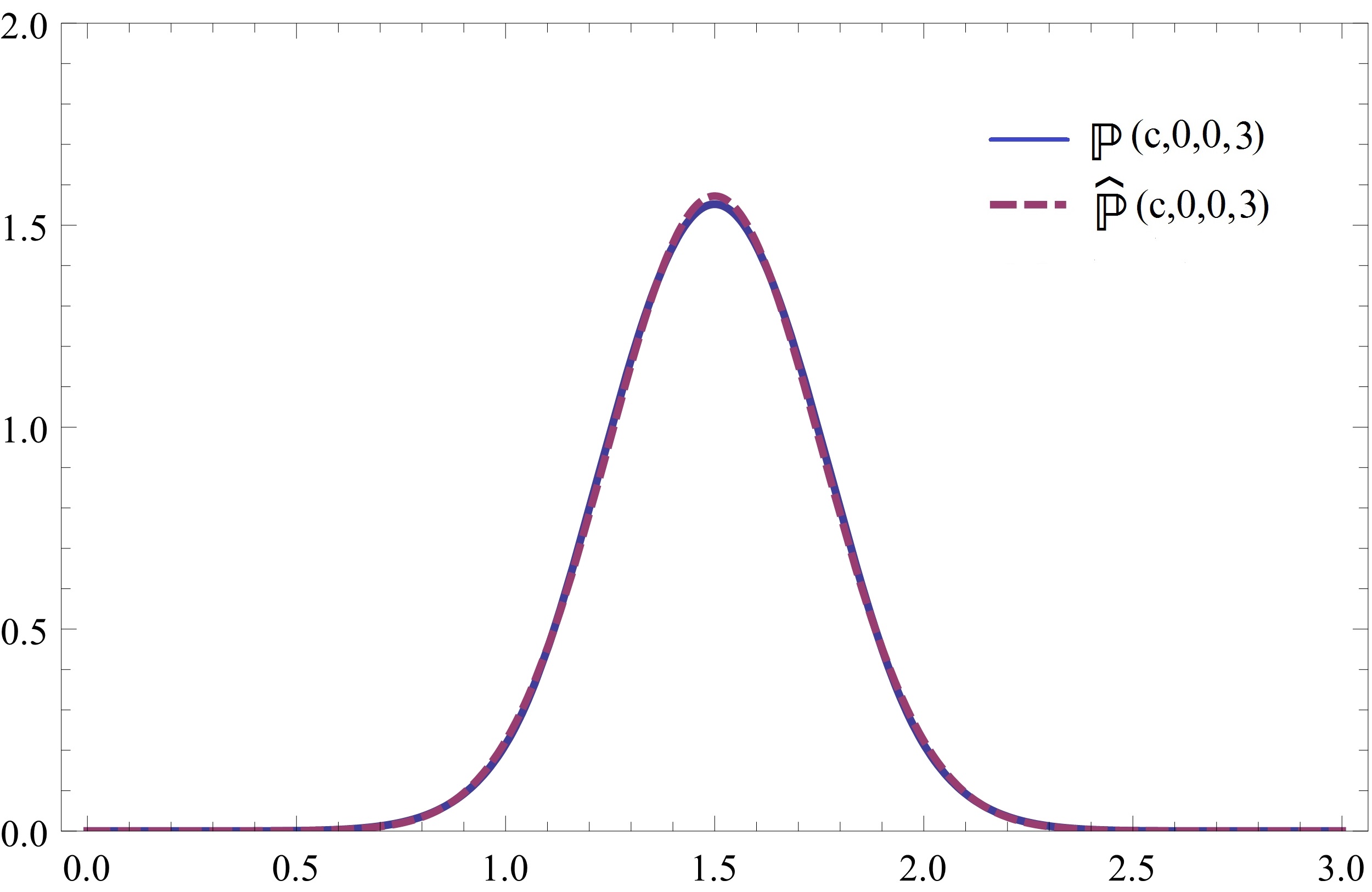} }
\subfigure[$n$=4]{
   \centering
  \includegraphics[width=7
cm,height=4cm]{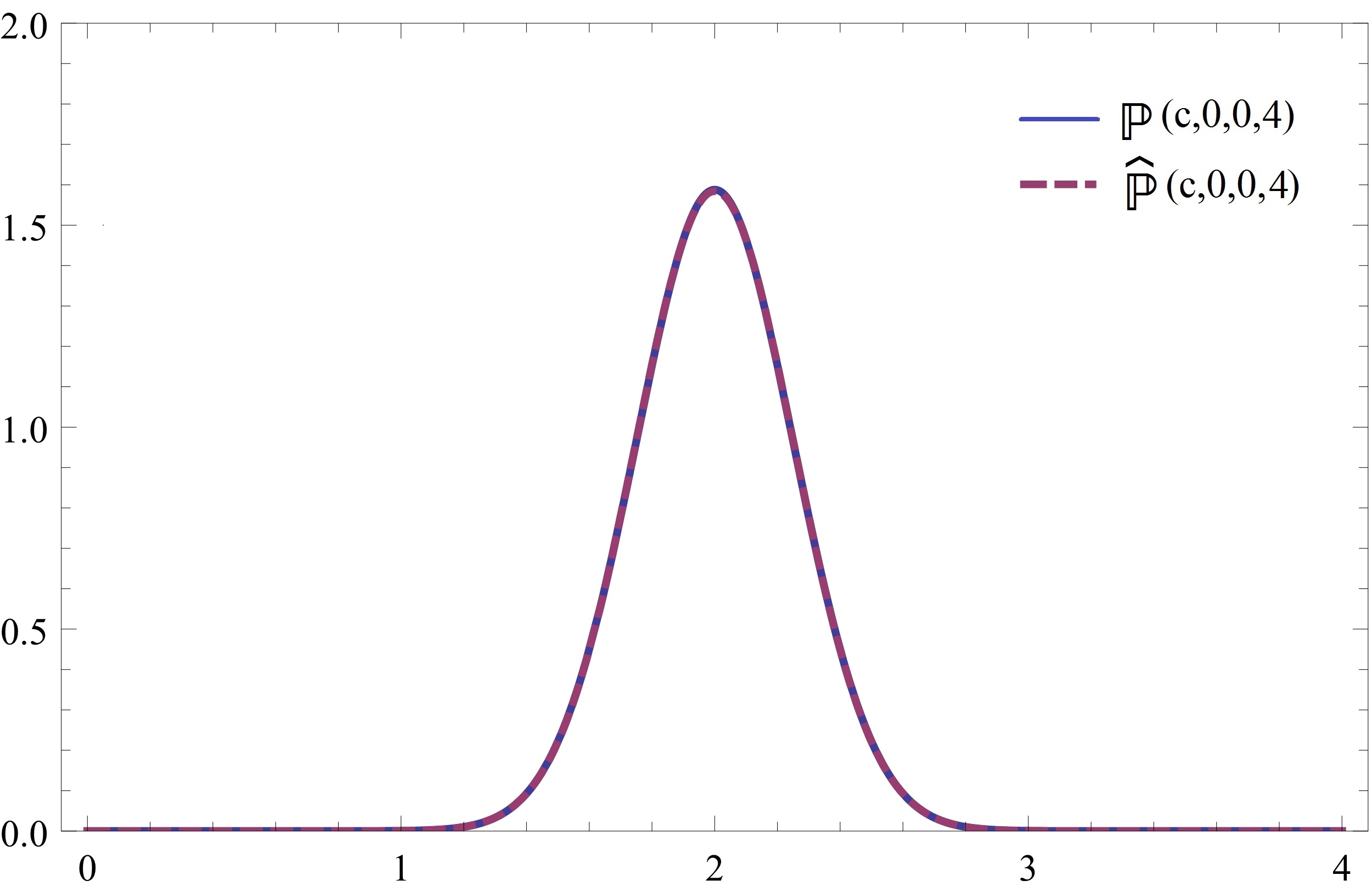} }
\subfigure[$n$=5]{
   \centering
  \includegraphics[width=7
cm,height=4cm]{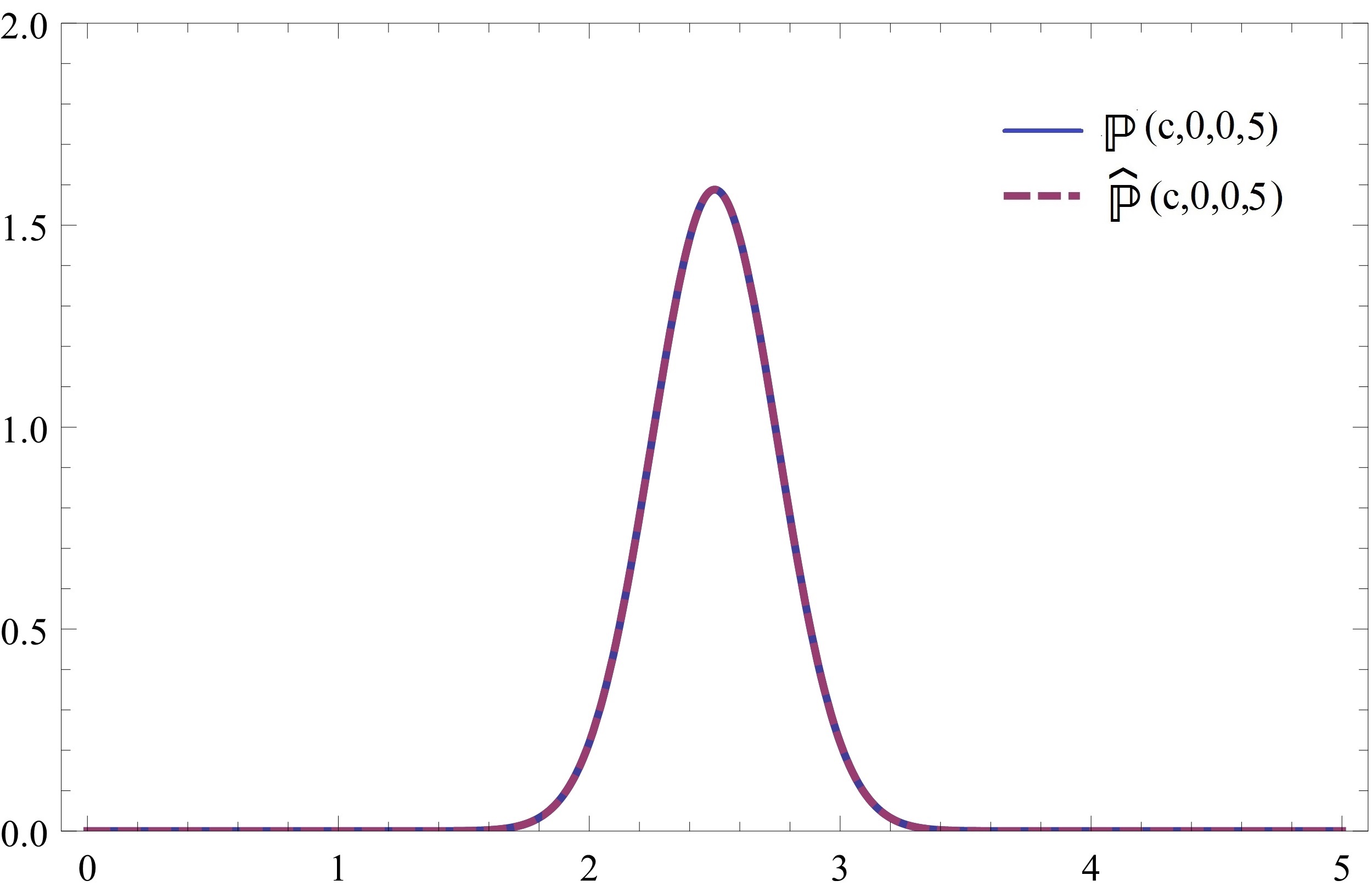} }
\caption{The coefficients distribution of $\mathbb{P}(c,0,0,n) $ and $\widehat{\mathbb{P}}(c,0,0,n) $}
\label{fig:1}
\end{figure}
\begin{figure}
\centering
\subfigure[$n$=2]{
 \centering
  \includegraphics[width=7
cm,height=3.9cm]{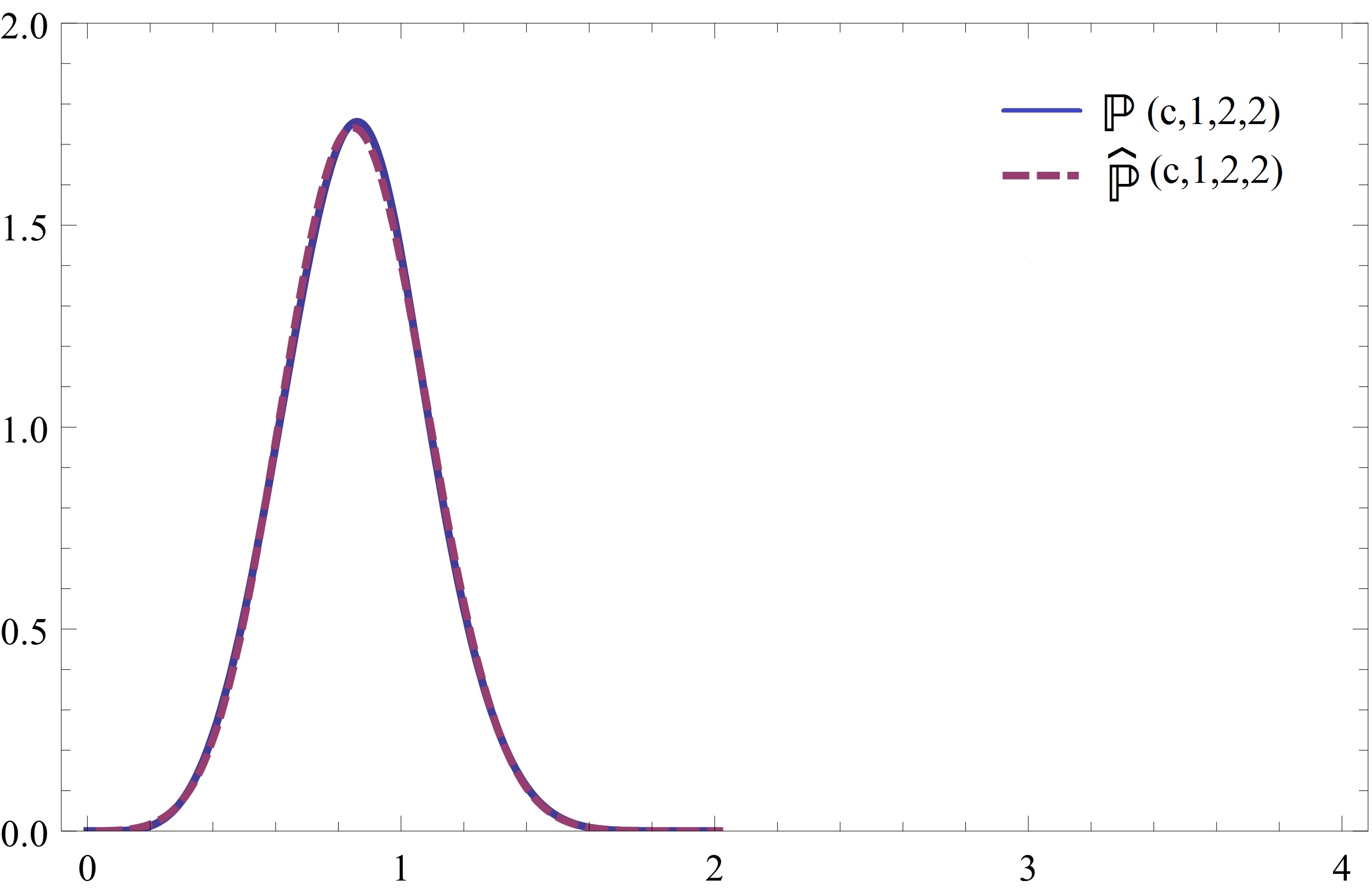}}
\subfigure[$n$=3]{
   \centering
  \includegraphics[width=7
cm,height=3.9cm]{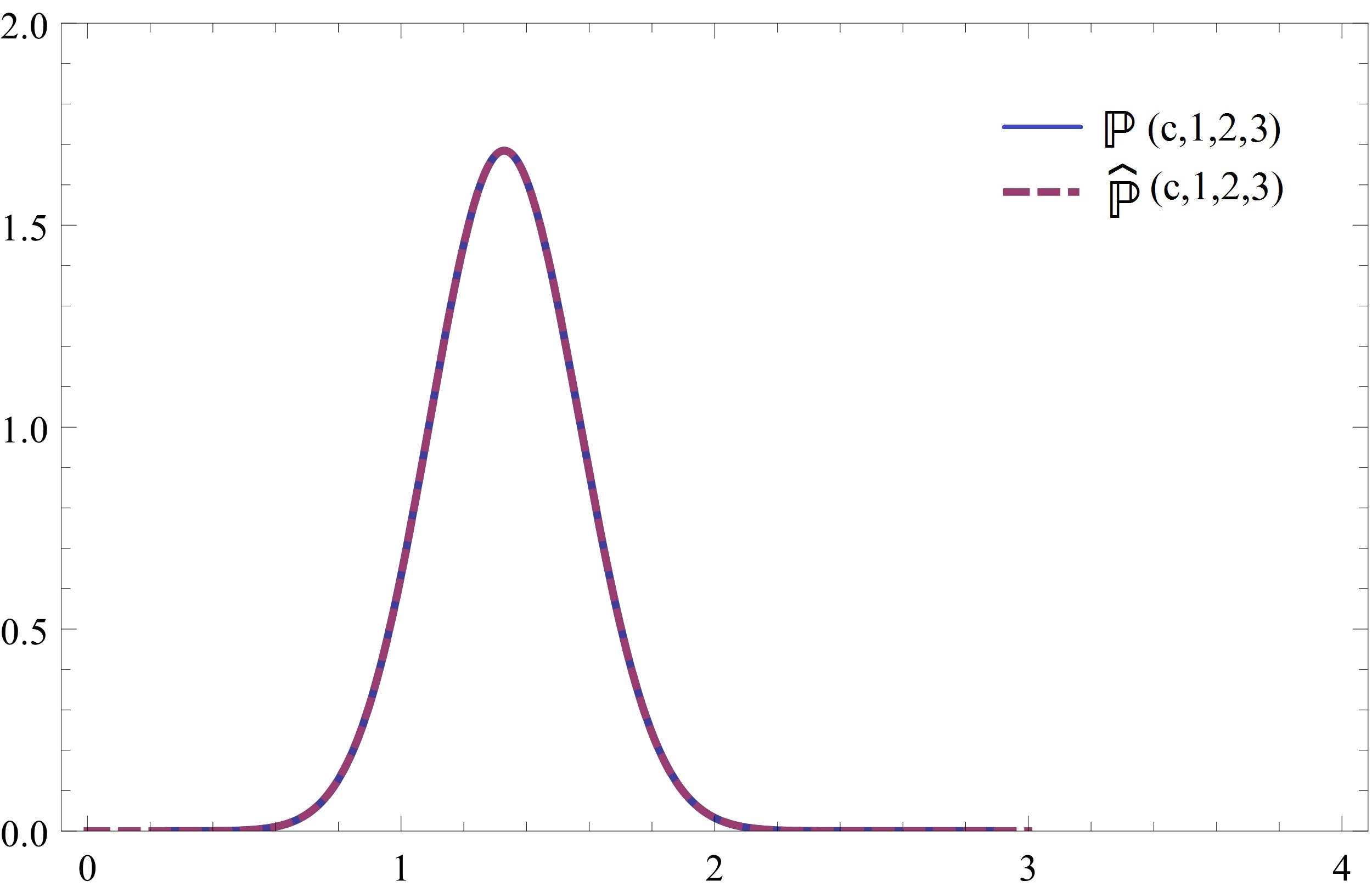} }
\subfigure[$n$=4]{
   \centering
  \includegraphics[width=7
cm,height=3.9cm]{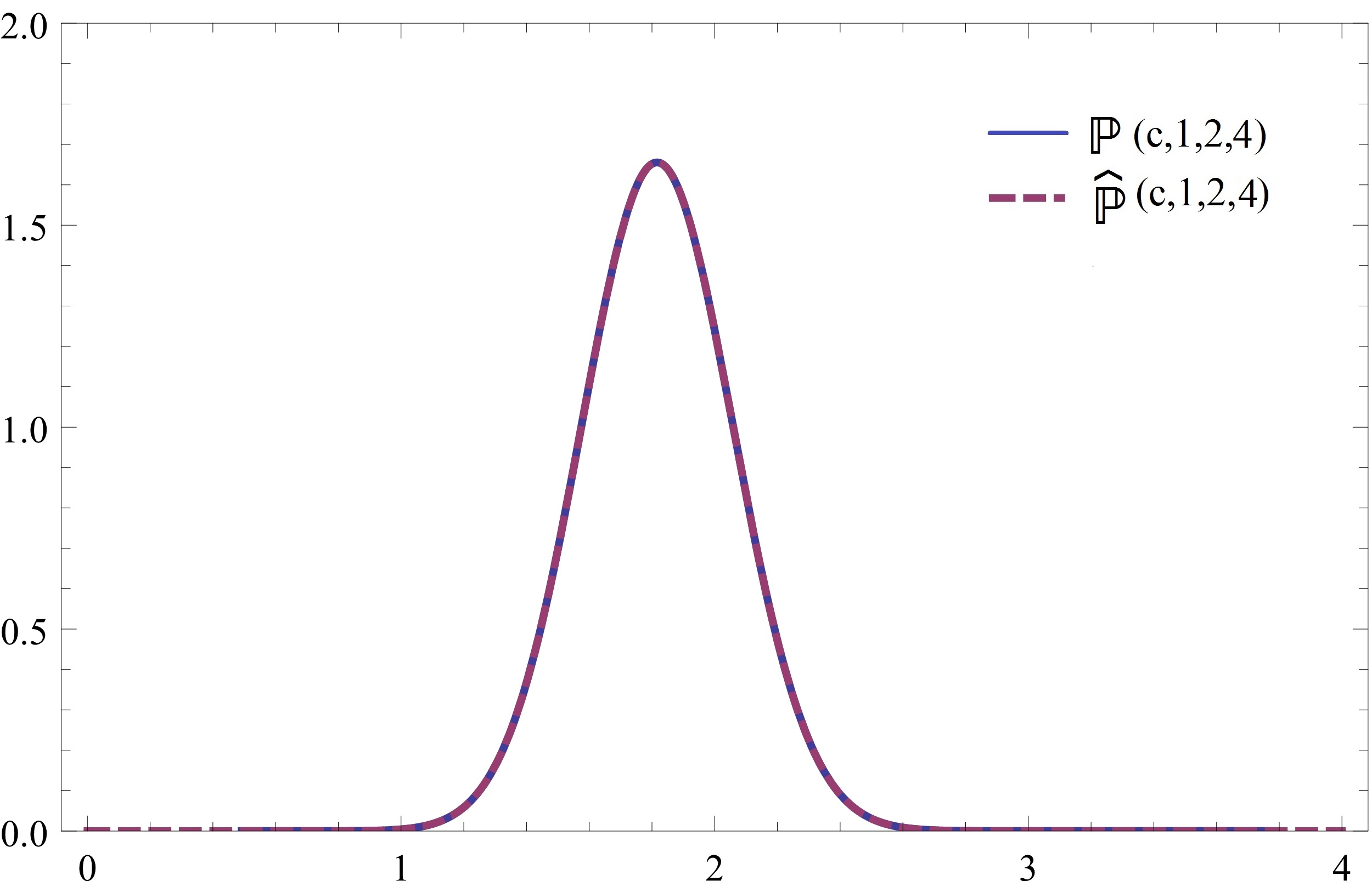} }
\caption{Compares $\mathbb{P}(c,1,2,n)$ with $\widehat{\mathbb{P}}(c,1,2,n)$}
\label{fig:2}
\end{figure}

\begin{rem}\begin{enumerate}
 Compare $\widehat{\mathbb{P}}(c,\alpha,\beta,n)$ with $\mathbb{P}(c,\alpha,\beta,n)$.  In Appendix A, we list the computed formulas for $\mathbb{P}(c,\alpha,\beta,n)$. In the Figures, we find there is almost coincidence when $n\geq3$, see Figure 1 ($\alpha=0$, $\beta=0$), Figure 2 ($\alpha=1$, $\beta=2$). The other cases exhibit similar behavior, so we infer that the approximation $\widehat{\mathbb{P}}(c,\alpha,\beta,n)$ is accurate when $n>3$.
 The expression $\widehat{\mathbb{P}}(c,0,0,n)$ here gives an easy way to characterise the coefficients of $\mathbb{P}(c,0,0,n)$ conjectured in \cite{KRRR}.
\end{enumerate}
\end{rem}

\begin{figure}
\centering
\subfigure[$n$=2]{
 \centering
  \includegraphics[width=7
cm,height=3.9cm]{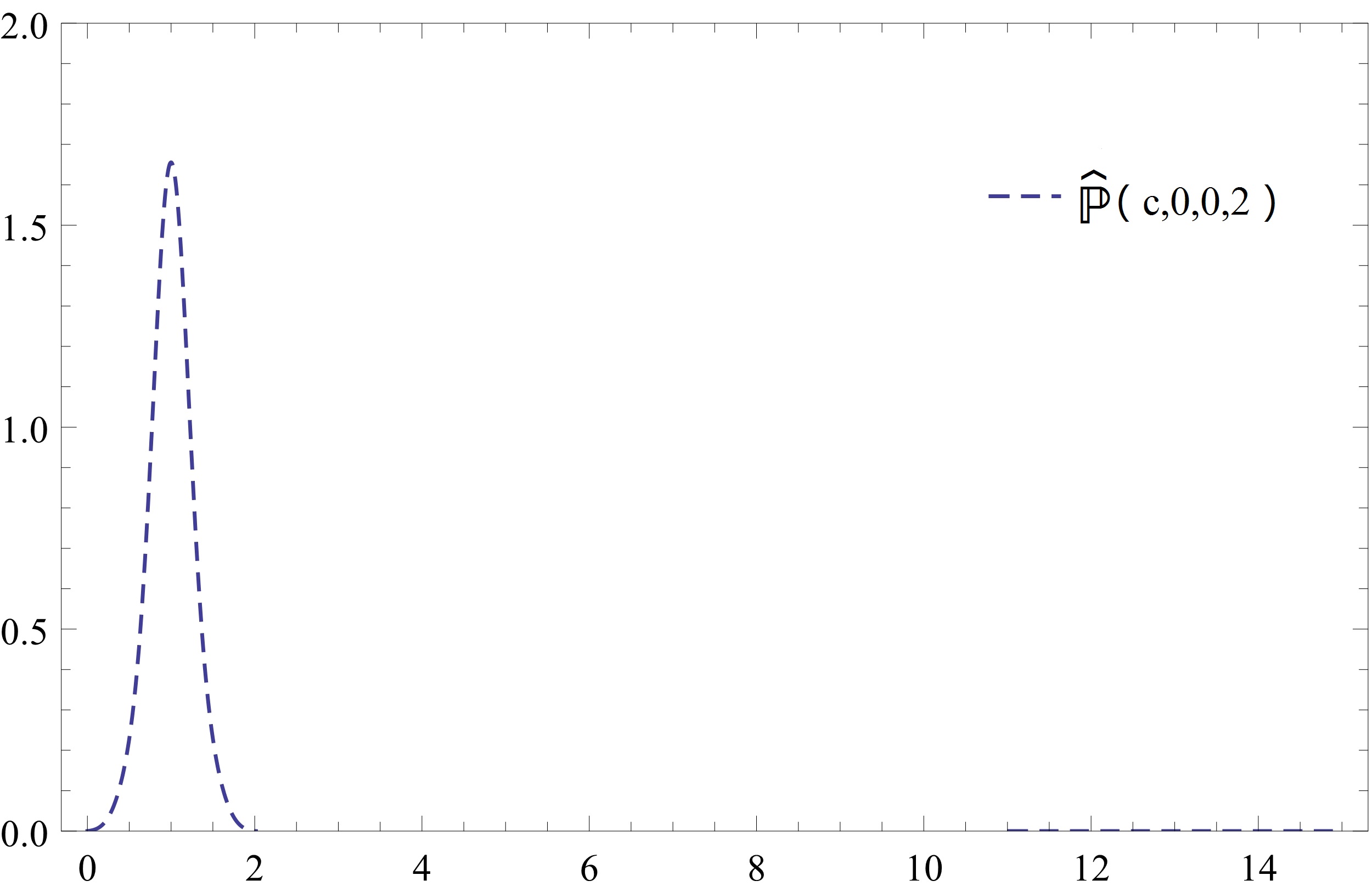}}
\subfigure[$n$=10]{
   \centering
  \includegraphics[width=7
cm,height=3.9cm]{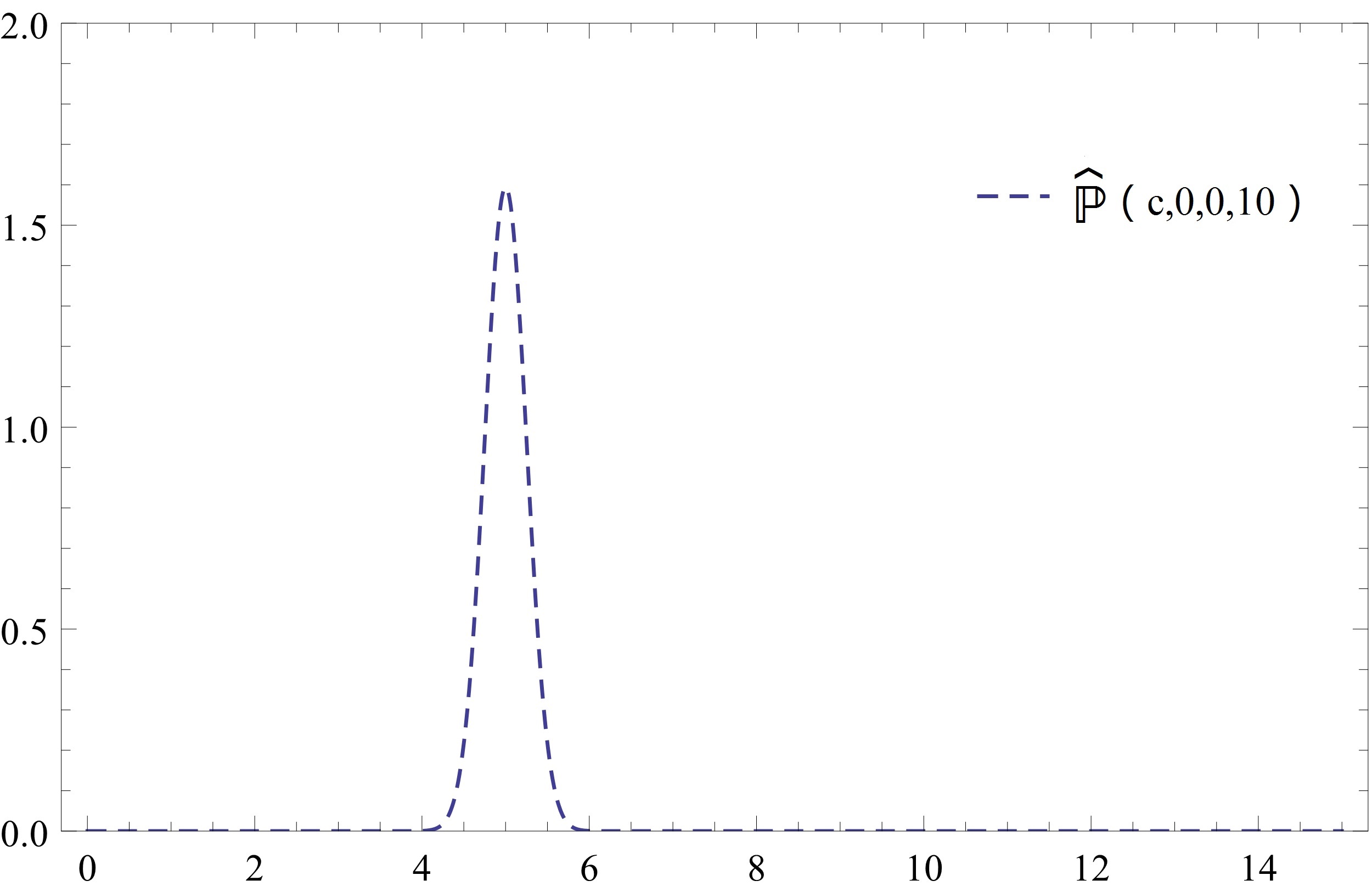} }
\subfigure[$n$=15]{
   \centering
  \includegraphics[width=7
cm,height=3.9cm]{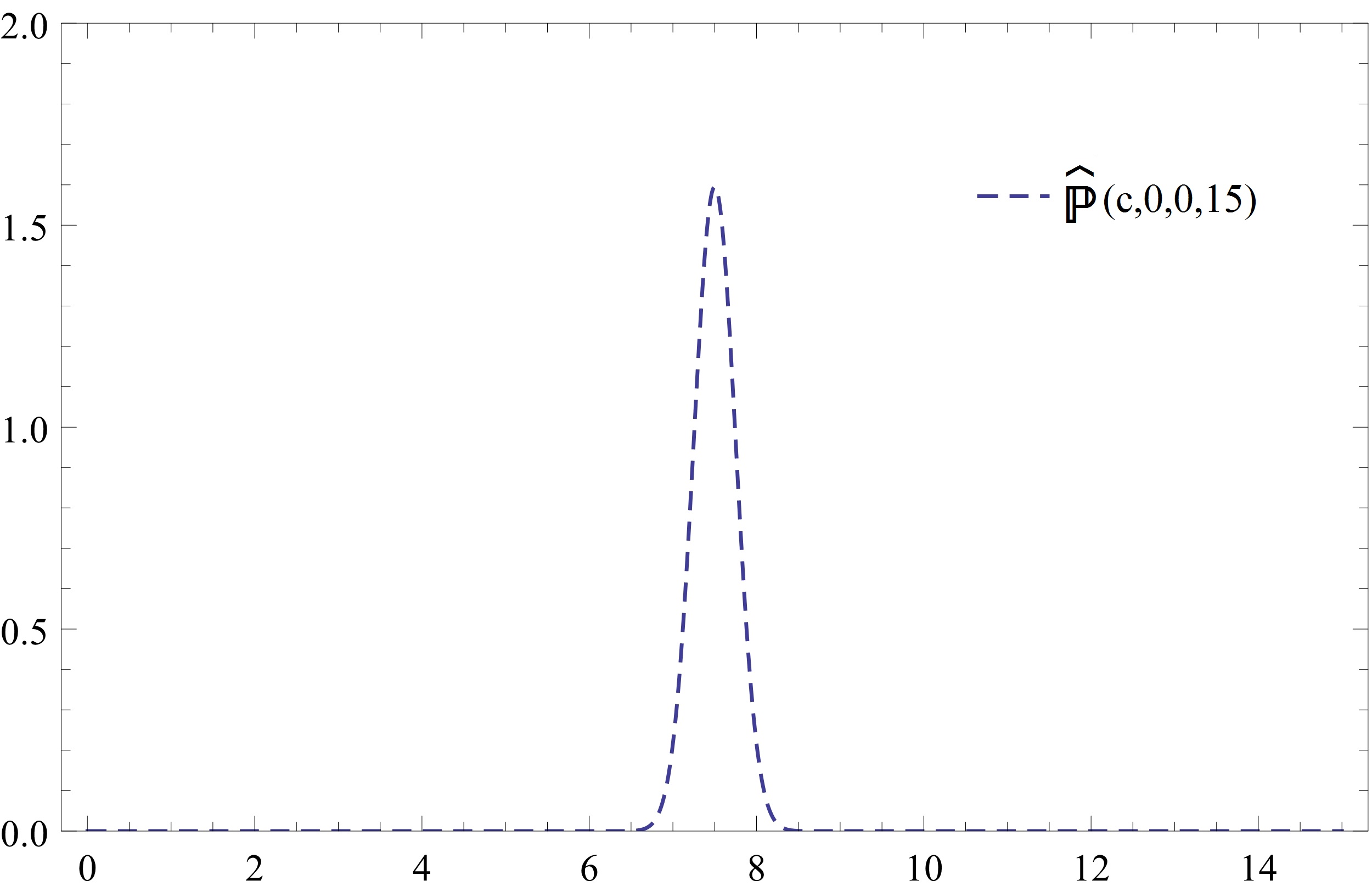} }
\caption{$\widehat{ \mathbb{P}}(c,0,0,n)$ changing with $n$}
\label{fig:2}
\end{figure}


\section{Uniform convexity}

Let $w_{0}(x)$ be a weight of the form $w_{0}(x)={\rm e}^{-v_{0}(x)}$ where $v_{0}$ is a continuously differentiable and convex real function such $v_{0}(x)\geq\log\left(1+x^{2}\right)$ for all $x$ such that $x^{2}\geq x_{0}$. Then the energy
\begin{equation}\label{6.1}
E_{v}(\rho)=\int_{-\infty}^{\infty}v_{0}(x)\rho(dx)+\int_{-\infty}^{\infty}\int_{-\infty}^{\infty}\log\frac{1}{|x-y|}\rho(dx)\rho(dy)
\end{equation}
is defined for all non-atomic probability measure $\rho$ that have finite logarithmic energy. Then the minimal energy $I_{v}=\inf\left\{E_{v}(\rho)\right\}$ is attained for a unique probability measure $\rho_{0}$ called the equilibrium measure that has compact support $[a,b]$, and $\rho_{0}$ is absolutely continuous with respect to Lebesgue measure, so $\rho_{0}(dx)=\sigma_{0}(x)dx$ for some probability density function $\sigma_{0}$. See \cite{DKM} and \cite{ST} for details. Also, there exist a constant $C_{v}$ such that $\sigma_{0}$ is determined almost everywhere by the inequality
\begin{equation}\label{6.2}
v_{0}(x)\geq2\int_{a}^{b}\log|x-y|\sigma_{0}(y)dy+C_{v}~~~~~~~~~(x\in\mathbf{R}),
\end{equation}
with equality if and only if $x\in(a,b)$.

The following is a complication of results which are known, or similar to those in the literature, See \cite{B2,LP}.

\begin{proposition}
Suppose that $p(x)=f(x)/\pi\sqrt{1-x^{2}}$ and $q(x)=g(x)/\pi\sqrt{1-x^{2}}$ are probability density functions on $[-1,1]$ where $f,g\in L^{2}(dx/\pi\sqrt{1-x^{2}})$ have Chebyshev expansions $f(x)=\sum_{k=0}^{\infty}b_{k}T_{k}(x)$.

(i) Then
\begin{equation}\label{6.8}
I(p,q)=\int_{-1}^{1}\int_{-1}^{1}\log\frac{1}{|x-y|}\left(p(x)-q(x)\right)\left(p(y)-q(y)\right)dxdy
\end{equation}
satisfies
\begin{equation}\label{6.9}
I(p,q)=\sum_{k=1}^{\infty}\left(a_{k}-b_{k}\right)^{2}/(2k).
\end{equation}

(ii)
In particular, the equilibrium measure satisfies $I(p,\sigma_{0})=E_{v}(p)-E_{v}(\sigma_{0})$.

(iii)
Suppose that $v$ is uniformly convex, so that $v''(x)\geq \gamma $ for all $x$ and some $\gamma>0$. Suppose that $p_{n}$ is a sequence of probability density functions as above such that $E_{v}(p_{n})\rightarrow E_{v}(\sigma_{0})$ as $n\rightarrow\infty$. Then $p_{n}\rightarrow\sigma_{0}$ in the weak topology.

(iv)
Suppose that $v$ is a polynomial. Then $\sigma_{0}(x)=\sum_{k=1}^{\infty}a_{k}T_{k}/\pi\sqrt{1-x^{2}}$ where only finitely many of the $a_{k}$ are non-zero.
\end{proposition}
\begin{proof}
({\rm \romannumeral1}) We substitute $x=\cos\theta$ and $y=\cos\phi$ and obtain
\begin{equation}\label{6.10}
\begin{split}
I(p,q)&=\frac{1}{\pi^{2}}\int_{-1}^{1}\int_{-1}^{1}\log\frac{1}{|x-y|}\left(f(x)-g(x)\right)\left(f(y)-g(y)\right)\frac{dxdy}{\sqrt{1-x^{2}}\sqrt{1-y^{2}}}\\
&=\sum_{k,l=1}^{\infty}\frac{(a_{k}-b_{k})(a_{l}-b_{l})}{4\pi^{2}}\int_{-\pi}^{\pi}\int_{-\pi}^{\pi}-\log|\cos\theta-\cos\phi|\cos k\theta\cos l\phi d\theta d\phi,
\end{split}
\end{equation}
where
\begin{equation}\label{6.11}
-\log|\cos\theta-\cos\phi|=-\log2-\log|\sin(\theta+\phi)/2|-\log|\sin(\theta-\phi)/2|,
\end{equation}
and
\begin{equation}\label{6.12}
\begin{split}
-\int_{-\pi}^{\pi}\log|\sin(\theta+\phi)/2|\cos k\theta\frac{d\theta}{2\pi}=-\int_{-\pi}^{\pi}\log|\sin\theta/2|\cos k\theta\frac{d\theta}{2\pi}\cos k\phi=\frac{\cos k\phi}{2k}
\end{split}
\end{equation}
by \cite{GradshteynRyzhik2007}. A similar identity holds for $\log|\sin(\theta-\phi)/2|$, and so by orthogonality, we obtain the stated result.

({\rm \romannumeral2}) This follows from the identity (\ref{6.2}) by a simple calculation.

({\rm \romannumeral3}) By uniform convexity, there exist $\gamma>0$ such that the Wasserstein transportation distance satisfies $W_{2}(p_{n},\sigma_{0})^{2}\leq\gamma I(p_{n},\sigma_{0})$, so $W_{2}(p_{n},\sigma_{0})\rightarrow0$ as $n\rightarrow\infty$, and $p_{n}\rightarrow\sigma_{0}$ weakly.

({\rm \romannumeral4}) By a formula of Tricomi,
\begin{equation}\label{6.13}
\begin{split}
v'(x)=&PV\int_{-1}^{1}\frac{2\sigma_{0}(y)}{x-y}dy\\
      &=\sum_{k=0}^{\infty}PV\int_{-1}^{1}\frac{2a_{k}T_{k}(y)}{\pi(x-y)\sqrt{1-y^{2}}}dy\\
      &=-2\sum_{k=1}^{\infty}a_{k}U_{k-1}(x),
\end{split}
\end{equation}
where $U_{k}(x)$ is Chebyshev's polynomial of the second kind of degree $k$. If $v$ is a polynomial, then the series has only finitely many non-zero terms.

Chen and Lawrence \cite{ChenLawrence1998} consider the effect of replacing $w_{0}(x)$ by $w_{0}(x){\rm e}^{-\lambda f(x)/n}$ or equivalently replacing $v_{0}(x)$ by $v_{0}(x)+\lambda f(x)/n$, where $f(x)$ is a bounded and continuous real function. The linear statistic has mean
\begin{equation}\label{6.14}
\begin{split}
n\int_{a}^{b}f(x)\sigma_{0}(x)dx
\end{split}
\end{equation}
and variance
\begin{equation}\label{6.15}
\frac{1}{2\pi^{2}}\int_{a}^{b}\frac{f(x)}{\sqrt{(b-x)(x-a)}}PV\int_{a}^{b}\frac{\sqrt{(b-y)(y-a)}f'(y)}{x-y}dydx.
\end{equation}
For a given $\sigma_{0}$, the possible values of the mean and variance are related, as in the following result.

By a simple scaling argument, we can replace $(a,b)$ by $(-1,1)$, and the standard deviation of $f(x)$ does not change if we add a constant to $f(x)$. Suppose therefore that $f(x)$ is an absolutely continuous real function on $(-1,1)$ such that $f(x)$ and $f'(x)$ are square integrable with respect to the Chebyshev weight $1/\pi\sqrt{1-x^{2}}$, such that
\begin{equation}\label{6.16}
\int_{-1}^{1}\frac{f(x)}{\pi\sqrt{1-x^{2}}}dx=0.
\end{equation}
For such $f(x)$, we consider the functional
\begin{equation}\label{6.17}
\Phi(f)=\frac{1}{\pi^{2}}\int_{-1}^{1}\frac{f(x)}{\sqrt{1-x^{2}}}PV\int_{-1}^{1}\frac{f'(y)\sqrt{1-y^{2}}}{x-y}dydx
\end{equation}
and aim to compute the Legendre transform of $\Phi$, as in
\begin{equation}\label{6.18}
\Phi^{*}(\sigma)=\sup_{f}\left\{\int_{-1}^{1}f(x)\sigma(x)dx-\Phi(f):\int_{-1}^{1}\frac{f(x)}{\pi\sqrt{1-x^{2}}}dx=0\right\}.
\end{equation}
The following result shows that $\Phi^{*}(\sigma)$ is a measure of the distance between $\sigma$ and the Chebyshev (arcsine) distribution on $[-1,1]$, in a metric associated with the periodic Sobolev space $H^{-1/2}$.

\end{proof}

\begin{proposition}
Let $h(x)=\pi\sigma(x)\sqrt{1-x^{2}}$, and let $g(x)=h(x)-\int_{-1}^{1}h(t)dt/\pi\sqrt{1-t^{2}}$. Then
\begin{equation}\label{6.19}
\Phi^{*}(\sigma)=\frac{1}{2\pi^{2}}\int_{-1}^{1}\int_{-1}^{1}\log\left|\frac{(1-x)(1-y)}{(x-y)^{2}}\right|\frac{g(x)g(y)}{\sqrt{1-x^{2}}\sqrt{1-y^{2}}}dxdy,
\end{equation}
or equivalently $\Phi^{*}(\sigma)=2^{-1}I(\sigma,1/\pi\sqrt{1-x^{2}})$ and equality is attained in the supremum if and only if
\begin{equation}\label{6.20}
g(x)=PV\frac{1}{\pi}\int_{-1}^{1}\frac{f'(y)\sqrt{1-y^{2}}}{x-y}dy
\end{equation}
almost everywhere.
\end{proposition}
\begin{proof}
We expand $f(x)$ and $h(x)$ in terms of Chebyshev polynomials of the first kind, so
\begin{equation}\label{6.21}
f(x)=\sum_{j=1}^{\infty}a_{j}T_{j}(x),~~~~~~~~~~h(x)=\sum_{j=0}^{\infty}b_{j}T_{j}(x).
\end{equation}
Then $f'(x)=\sum_{j=1}^{\infty}j a_{j} U_{j-1}(x)$ where $U_{j}(x)$ are the Chebyshev polynomials of the second kind as in \cite{GradshteynRyzhik2007}, and by a formula of Tricomi
\begin{equation}\label{6.22}
PV\frac{1}{\pi}\int_{-1}^{1}\frac{f'(y)\sqrt{1-y^{2}}}{x-y}dy=\sum_{j=1}^{\infty}ja_{j}T_{j}(x)
\end{equation}
by \cite{GradshteynRyzhik2007}. Then
\begin{equation}\label{6.23}
\frac{1}{\pi^{2}}\int_{-1}^{1}\frac{f(x)}{\sqrt{1-x^{2}}}PV\int_{-1}^{1}\frac{f'(y)\sqrt{1-y^{2}}}{x-y}dydx=\sum_{j=1}^{\infty}ja_{j}\int_{-1}^{1}\frac{f(x)T_{j}(x)}
{\pi\sqrt{1-x^{2}}}dx=\sum_{j=1}^{\infty}\frac{ja_{j}^{2}}{2};
\end{equation}
hence $\Phi (f)=\sum_{j=1}^\infty ja_j^2/2$. We deduce that
\begin{equation}\label{6.24}
\Phi^* (h)=\sup\Bigl\{ \sum_{j=1}^{\infty}b_{j}a_{j}-\sum_{j=1}^{\infty}\frac{ja_{j}^{2}}{2}\Bigr\}\leq\sum_{j=1}^{\infty}\frac{b_{j}^{2}}{2j}
\end{equation}
with equality attained if and only if $b_{j}=ja_{j}$ for all $j=1,2,\ldots$. Hence $\Phi^{*}(h)=\sum_{j=1}^{\infty}b_{j}^{2}/2j$, which we can compare with the formula (\ref{6.19}). We now identify this series with a double integral. We can write $g(x)=\sum_{k=1}^{\infty}b_{k}T_{k}(x)$; then by another formula of Tricomi \cite{GradshteynRyzhik2007}, the transform
\begin{equation}\label{6.25}
\tilde{g}(x)=PV\frac{1}{\pi}\int_{-1}^{1}\frac{g(y)dy}{(y-x)\sqrt{1-y^{2}}}
\end{equation}
satisfies
\begin{equation}\label{6.26}
\tilde{g}(x)=\sum_{j=1}^{\infty}b_{j}U_{j-1}(x);
\end{equation}
and taking the integral of the series, we obtain
\begin{equation}\label{6.27}
-\int_{x}^{1}\tilde{g}(t)dt=\sum_{j=1}^{\infty}\frac{b_{j}}{j}(T_{j}(x)-T_{j}(1)),
\end{equation}
in which $T_{j}(1)=1$ for all $j=1,2,\ldots$ by \cite{GradshteynRyzhik2007}. Then
\begin{equation}\label{6.28}
\frac{-1}{\pi}\int_{-1}^{1}g(x)\int_{x}^{1}\tilde{g}(t)dt\frac{dx}{\pi\sqrt{1-x^{2}}}=\sum_{j=1}^{\infty}\frac{b_{j}^{2}}{2j}.
\end{equation}
We can also write
\begin{equation}\label{6.29}
-\int_{x}^{1}\tilde{g}(t)dt=\int_{-1}^{1}\log\left|\frac{1-y}{x-y}\right|\frac{g(y)dy}{\pi\sqrt{1-y^{2}}},
\end{equation}
hence by symmetrizing the variables, we have
\begin{equation}\label{6.30}
\frac{-1}{\pi}\int_{-1}^{1}g(x)\int_{x}^{1}\tilde{g}(t)dt\frac{dx}{\pi\sqrt{1-x^{2}}}=\frac{1}{2\pi^{2}}\int_{-1}^{1}\int_{-1}^{1}\log\left|\frac{(1-x)(1-y)}
{(x-y)^{2}}\right|\frac{g(x)g(y)}{\sqrt{1-x^{2}}\sqrt{1-y^{2}}}dxdy.
\end{equation}
This identifies $\Phi^{*}(\sigma)$ with the double integral. Also, the supremum is attained if and only if $f(x)$ and $g(x)$ have $b_{j}=ja_{j}$ for $j=1,2,\ldots$, so the above integral equation holds almost everywhere.
\end{proof}
\begin{example}
Starting with the classical Jacobi weight on $[-1,1]$, we can introduce a limiting density, which lives on a proper subinterval $(A_{n},B_{n})$. As in (\ref{4.14}), let $\rho$ be the limiting density
\begin{equation*}
\rho(y)=\frac{1}{\pi}\frac{n+(\alpha+\beta)/2}{1-y^{2}}\sqrt{(B_{n}-y)(y-A_{n})}~~~~~~(y\in (A_{n},B_{n})).
\end{equation*}
We suppose that $\alpha=\beta$, so $A_{n}=-B_{n}=-K$, and then we rescale $[-1,1]$ to $[-1/K,1/K]$, and $[-K,K]$ to $[-1,1]$, to obtain the probability density function
\begin{equation*}
\sigma(x)=\kappa\frac{\sqrt{1-x^{2}}}{1-K^{2}x^{2}}~~~~~~~(x\in[-1,1])
\end{equation*}
where $\kappa>0$ and $0<K<1$ are constants. In view of the Proposition, $\Phi^{*}(\sigma)$ is a measure of the distance between $\sigma$ and the Chebyshev (arcsine) distribution on $[-1,1]$; for $K=1$, we indeed have the arcsine distribution, whereas for $K=0$, we have the semicircular law. We compute
\begin{equation*}
h(x)=\pi\sqrt{1-x^{2}}\sigma(x)=\frac{\kappa\pi(1-x^{2})}{1-K^{2}x^{2}},
\end{equation*}
and then introduce the Chebyshev coefficients of $h$. We have
\begin{equation*}
a_{n}=\int_{-1}^{1}\frac{h(x)T_{n}(x)dx}{\pi\sqrt{1-x^{2}}}=\frac{\pi\kappa}{2\pi}\int_{0}^{2\pi}\left[\frac{1}{K^{2}}+\frac{K^{2}-1}{2K^{2}}\left(\frac{1}
{1-K\cos\theta}+\frac{1}{1+K\cos\theta}\right)\right]{\rm e}^{{\rm i}n\theta}d\theta.
\end{equation*}
We can replace this by a contour integral around the unit circle, so by an elementary calculus of residues, we obtain
\begin{equation*}
a_{n}=-\pi\kappa\frac{\sqrt{1-K^{2}}}{K^{n+2}}\left(-1+\sqrt{1-K^{2}}\right)^{n} ~~~~~~~(n=1,2\ldots),
\end{equation*}
and
\begin{equation*}
a_{0}=\pi\kappa\frac{1-\sqrt{1-K^{2}}}{K^{2}},
\end{equation*}
where $\pi\kappa=K^{2}/\left(1-\sqrt{1-K^{2}}\right)$ since $\sigma$ is a probability density function. Hence
\begin{equation*}
\Phi^{*}(\sigma)=\sum_{n=1}^{\infty}\frac{a_{n}^{2}}{2n}=\frac{-\pi^{2}\kappa^{2}(1-K^{2})}{2K^{4}}\log\left[1-\left(\frac{-1+\sqrt{1-K^{2}}}
{K}\right)^{2}\right].
\end{equation*}
When $\alpha=\beta=\lambda-1/2$, the corresponding system of orthogonal polynomials is given by the Gegenbauer (ultraspherical) polynomials $(G_{n}^{\lambda})_{n=0}^{\infty}$ which satisfy, for $L_{\lambda}f=-\left(1-x^{2}\right)f''(x)+(2\lambda+1)xf'(x)$, the eigenfunction equation
\begin{equation*}
L_{\lambda}G_{n}^{\lambda}=n(n+2\lambda)G_{n}^{\lambda}.
\end{equation*}
\end{example}
We conclude this section with a result concerning fluctuations. Suppose that $v_0$ is uniformly convex, so that $v_0''(x)\geq \gamma $ for all $x$ and some $\gamma>0$. Let $V(X)={\rm trace}~v_0(X)$, and $\hat \mu_n(dX)=Z^{-1}{\rm e}^{-nV(X)}dX$; note that we use a different scaling from equation (\ref{7.3}).
Let $f: \mathbf{R}\rightarrow\mathbf{R}$ be a compactly supported smooth function, and introduce the linear statistic $F: M_{n}^{h}(\mathbf{C})\rightarrow\mathbf{R}$ associated with $f$ by $F(X)={\rm trace}~f(X)$. The fluctuations of $F$ are
\begin{equation}\label{6.3}
FL(X)=F(X)-\int_{M_{n}^{h}(\mathbf{C})} F(Y)\hat\mu_{n}(dY)~~~~~~~~~(X\in M_{n}^{h}(\mathbf{C})).
\end{equation}
\begin{proposition} Then
\begin{equation}\lim\sup_{n\rightarrow\infty}\int_{M_{n}^{h}(\mathbf{C})} \left(FL(X)\right)^{2}\hat\mu_{n}(dX)\leq\frac{1}{\gamma}\int_{a}^{b}f'(x)^{2}\sigma_{0}(x)dx \end{equation}
and
\begin{equation}\int_{M_{n}^{h}(\mathbf{C})}  \exp \bigl( tFL(X))\hat\mu_n(dX)\leq \exp\Bigl( {\frac{t^2}{2\Vert f'\Vert_\infty}}\Bigr)\qquad (t\in {\mathbf{R}}).
\end{equation}
\end{proposition}
\begin{proof} By the Rayleigh--Ritz formula (\ref{4.2}), we have
\begin{equation}\label{6.4}
\langle {\rm Hess~V, Y\otimes Y}\rangle\geq\gamma\langle Y,Y\rangle~~~~~~(Y\in M_{n}^{h}(\mathbf{C})),
\end{equation}
so $V: M_{n}^{h}(\mathbf{C})\rightarrow\mathbf{R}$ is uniformly convex. See \cite{B1}. By the Bakry--Emery criterion, $\hat\mu_{n}$ satisfies a logarithmic Sobolev inequality in the form
\begin{align}\label{6.5}
\int_{M_{n}^{h}(\mathbf{C})}  G(X)^{2}\log G(x)^{2}\hat\mu_{n}(dX)&\leq\int_{M_{n}^{h}(\mathbf{C})}  G(X)^{2}\hat\mu_{n}(dX)\log\int_{M_{n}^{h}(\mathbf{C})}  G(X)^{2}\hat\mu_{n}(dX)\nonumber\\
&+\frac{2}{n\gamma}\int_{M_{n}^{h}(\mathbf{C})} \Vert\nabla G(X)\Vert^{2}\hat\mu_{n}(dX).
\end{align}
By applying this inequality to $G(X)=1+tFL(X)$ with small real $t$, we deduce that
\begin{equation}\label{6.6}
\int_{M_{n}^{h}(\mathbf{C})} \left(FL(X)\right)^{2}\hat\mu_{n}(dX)\leq\frac{1}{n\alpha}\int_{M_{n}^{h}(\mathbf{C})} \Vert \nabla F(X) \Vert^{2}\hat\mu_{n}(dX).
\end{equation}
The right-hand side converges as $n\rightarrow\infty$, so
\begin{equation}\label{6.7}
\frac{1}{n\alpha}\int_{M_{n}^{h}(\mathbf{C})} \left({\rm trace}f'(X)\right)^{2}\hat\mu_{n}(dX)\rightarrow\frac{1}{\gamma}\int_{a}^{b}f'(x)^{2}\sigma_{0}(x)dx .
\end{equation}
Finally, we use (6.6) from \cite{B1} to obtain the stated concentration inequality.
\end{proof}

 \section{Appendix A: On $\mathbb{P}(c,\alpha,\beta, n)$ for finite $n$.}
Note that the PDF of the center of mass of the unitary Jacobi ensemble is
\begin{equation*}
\mathbb{P}(c,\alpha,\beta,n)=\frac{1}{2\pi D_n(0,\alpha,\beta)}\int_{-\infty}^\infty d\lambda\; \mathrm{e}^{-{\rm i} c\lambda}\frac{1}{n!}\int_{[0,1]^n}\Delta_n(\vec{x}){}^2 \prod_{l=1}^n x_l^\alpha (1-x_l)^\beta\mathrm{e}^{{\rm i} x_l\lambda} dx_l,
\end{equation*}
The Paley-Wiener theorem reads,
\\ \\
\noindent{\bf Theorem A.}(\cite{SW}, p.108) Suppose $F\in L^{2}(-\infty,\infty)$. Then $F(\xi)$, the Fourier transform of the function $f(x)$, vanishing outside $[-(\sigma/2\pi),\sigma/2\pi]=:[-\tau,\tau]$, i.e. $F(\xi):=\int_{-\infty}^{\infty}f(x){\rm e}^{-2\pi {\rm i}x\xi}dx=\int_{-\tau}^{\tau}f(x){\rm e}^{-2\pi {\rm i}x\xi}dx$, $x\in\mathbb{R}$ if and only if $F$ is an entire function of exponential type $\sigma$, $|F(\xi)|\leq A{\rm e}^{\sigma|\xi|}$, $\xi\in \mathbb{C}$, $\sigma>0$ and $A$ is a constant.
\\\\
\noindent Based on the above theorem, we have
\begin{lemma}\label{Lemma}
The Fourier transform of our $D_{n}(-{\rm i\lambda}, \alpha,\beta, n)$ given in (\ref{zhu}), denoted by $\mathbb{P}(c, \alpha,\beta, n)$, is supported in the interval $[0, n]$.
\end{lemma}
\begin{proof}
Consider a general case
$$D(z, [w], n): ={{1}\over{n!}}\int_{[0,1]^n} |\Delta_n(\vec{x})|^{\nu} \prod_{k=1}^n w(x_k) {\rm e}^{-{\rm i} x_k z} dx_k.$$
where $\nu>0$, and $w(x)$ is any smooth positive function integrable over $(0,1)$.
\noindent
Then
$${\rm e}^{{\rm i} nz/2}D(z, [w], n)={{1}\over{n!}}\int_{[0,1]^n} |\Delta_n(\vec{x})|^{\nu}
\exp\bigl( -{\rm i}z\sum_{k=1}^n (x_k-1/2)\bigr)\prod_{k=1}^n w(x_k)  dx_k$$
\noindent
where $-n/2\leq \sum_{k=1}^n (x_k-1/2)\leq n/2$, so ${\rm e}^{{\rm i} nz/2}D(z, [w], n)$ is entire, and there exists $C$ such that
$$\vert {\rm e}^{{\rm i} nz/2}D(z, [w], n)\vert\leq C{\rm e}^{n\vert y\vert /2}\qquad (z=x+{\rm i} y\in { \mathbb{C}}).$$
\noindent
Hence by the Paley--Wiener theorem from Stein and Weiss \cite{SW}, there exists a distribution $\varphi$ on $[-n/2,n/2]$ such that
$${\rm e}^{{\rm i}nz/2}D(z, [w], n)=\int_{[-n/2,n/2]} {\rm e}^{-{\rm i} zt}\varphi (t)\, dt,$$
\noindent so
$$D(z,[w], n)=\int_{[0,n]} {\rm e}^{-{\rm i}zs}\varphi (s-n/2)\, ds$$
\noindent where $\varphi (s-n/2)$ is a distribution supported on $[0,n]$.
\\\\
\noindent
So for our problem, $w(x)=x^{\alpha}(1-x)^{\beta},~ x\in(0,1)$, and $\nu=2$ follows.
\end{proof}

\noindent There follow formulas for $\mathbb{P}(c,\alpha ,\beta ,n)$ with $n=2,\dots ,5$ and three cases of  $\alpha$ and $\beta$.\\
\\
Case \uppercase\expandafter{\romannumeral1}\::\:\: $\alpha=0$, $\beta=0$
 \begin{align*}
\mathbb{P}(c,0,0,2)&=
\begin{cases}
2 c^3, ~~~~~~~~~~~~~~~~~ 0\leq c\leq 1, \\
2(2-c)^3,\ \
~~~~~~~ 1< c\leq 2;
\end{cases}\\
\mathbb{P}(c,0,0,3)&=
\begin{cases}
\frac{3c^8}{14}, ~~~~~~~~~~~~~~~~~~~~~~~~~~~~~~~~~~~~~~~~~~~~~~~~~~~~~~~~~~~ 0\leq c\leq 1,\\
\frac{3}{14}(-2c^8+24c^7-252c^6+1512c^5-4830c^4\\+8568c^3-8484c^2
+4392c-927),\ \
~~~~~~~~~~~~~~~~~ 1< c\leq2, \\
\frac{3(3-c)^8}{14},\ \
~~~~~~~~~~~~~~~~~~~~~~~~~~~~~~~~~~~~~~~~~~~~~~~~~~~~ 2< c\leq 3;
\end{cases}
\end{align*}
\begin{align*}
\mathbb{P}(c,0,0,4)&=
\begin{cases}
\frac{2c^{15}}{3003}, ~~~~~~~~~~~~~~ 0\leq c\leq 1,\\
\frac{2}{3003}\xi_{42}(c),~~~~~~~ 1< c\leq2, \\
\frac{2}{3003}\xi_{43}(c),~~~~~~~ 2< c\leq 3,\\
\frac{2(4-c)^{15}}{3003},~~~~~~~~~~3< c\leq4.
\end{cases}
\mathbb{P}(c,0,0,5)&=
\begin{cases}
\frac{5 c^{24}}{140229804}, ~~~~~~~~~~~~~~~ 0\leq c\leq 1,\\
\frac{5}{140229804}\xi_{52}(c),~~~~~~~~1< c\leq 2, \\
\frac{5}{140229804}\xi_{53}(c),~~~~~~~~2< c\leq3,\\
\frac{5}{140229804}\xi_{54}(c),~~~~~~~~3< c\leq4,\\
\frac{5(5-c)^{24}}{140229804},~~~~~~~~~~~~~~~4< c\leq5.
\end{cases}
\end{align*}
where
\begin{align*}
  \xi_{42}(c) =&-3 c^{15}+60 c^{14}-1680 c^{13}+29120 c^{12}-294840 c^{11}+1873872 c^{10}-7927920 c^9\\&+23268960 c^8-48674340 c^7+73653580 c^6-80912832 c^5+63969360 c^4\\&-35497280 c^3+13131720 c^2-2910240 c+292464, \\ \\
   \xi_{43}(c) =&3 c^{15}-120 c^{14}+3360 c^{13}-58240 c^{12}+644280 c^{11}-4948944 c^{10}+28428400 c^9\\&-128700000 c^8+470398500 c^7-1381480100 c^6+3179336160 c^5-5531176560 c^4\\&+6950332480 c^3-5910494520 c^2+3031004640 c-705916304,\\
      \xi_{52}(c) =&-4 c^{24}+120 c^{23}-6900 c^{22}+253000 c^{21}-5578650 c^{20}+79695000 c^{19}-785367660 c^{18}\\&+5598232200 c^{17}-29915282925 c^{16}+123134189200 c^{15}-398517412920 c^{14}\\&+1029946456560 c^{13}-2149736416100 c^{12}+3651921075600 c^{11}-5072249298600 c^{10}\\&+5768661885360 c^9-5363308269495 c^8+4055447662200 c^7-2470634081300 c^6\\&+1194550480200 c^5-447845361810 c^4+125530048600 c^3\\&-24758793900 c^2+3065085000 c-179192775,   \\ \\
   \xi_{53}(c) =&6 c^{24}-360 c^{23}+20700 c^{22}-759000 c^{21}+17798550 c^{20}-292215000 c^{19}+3673797820 c^{18}\\&-38235839400 c^{17}+347123925225 c^{16}-2790376974000 c^{15}+19589544660840 c^{14}\\&-117507788504400 c^{13}+592028782736300 c^{12}-2479096272534000 c^{11}\\&+8573537591434200 c^{10}-24367026171730000 c^9+56603181050415945 c^8\\&-106665764409131400 c^7+161304132700472300 c^6-192656070655587000 c^5\\&+177464649282553710 c^4-121528934511474600 c^3+58223870087874900 c^2\\&-17407730744067000 c+2443806916000825,
   \end{align*}
\begin{align*}
   \xi_{54}(c) =&-4 c^{24}+360 c^{23}-20700 c^{22}+759000 c^{21}-18861150 c^{20}+345345000 c^{19}-4991492660 c^{18}\\&+59676982200 c^{17}-604502001675 c^{16}+5220961534800 c^{15}-38343917872920 c^{14}\\&+238359873297840 c^{13}-1250073382257700 c^{12}+5522495132708400 c^{11}\\&-20539021982760600 c^{10}+64263112978594640 c^9-168820549421134545 c^8\\&+370693368908418600 c^7-674525363862958300 c^6+1002229415508043800 c^5\\&-1187187920423969310 c^4+1078975874367012600 c^3-706068990841773900 c^2\\&+295689680026989000 c-59394510856327775.
\end{align*}
Case \uppercase\expandafter{\romannumeral2}\::\:\: $\alpha=1$, $\beta=1$
 \begin{align*}
\mathbb{P}(c,1,1,2)&=
\begin{cases}
\frac{12}{7} c^5 \left(c^2-7 c+7\right), ~~~~~~~~~~~~~~~~~~~~~~~~~~~~~~~~~~~~~~~~~~~~ 0\leq c\leq 1, \\
\frac{12}{7} (2-c)^5 (c^2 +3c-3),\ \
~~~~~~~~~~~~~~~~~~~~~~~~~~~~~~~~~~~ 1< c\leq 2;
\end{cases}\\
\mathbb{P}(c,1,1,3)&=
\begin{cases}
\frac{10}{1001}c^{11} \left(-c^3+21 c^2-91 c+91\right), ~~~~~~~~~~~~~~~~~~~~~~~~~~ 0\leq c\leq1,\\
\frac{10}{1001}\zeta^1_{32},\ \
~~~~~~~~~~~~~~~~~~~~~~~~~~~~~~~~~~~~~~~~~~~~~~~~~~~~~~~~ 1< c\leq 2, \\
\frac{10}{1001}(c-3)^{11} \left(-c^3-12 c^2+8 c+20\right),\ \
~~~~~~~~~~~~~~~~~~ 2< c\leq 3;
\end{cases}\\
\mathbb{P}(c,1,1,4)&=\begin{cases}
\frac{10}{11685817}c^{19} \left(c^4-46 c^3+506 c^2-1771 c+1771\right), ~~~~~~~ 0\leq c\leq 1,\\
\frac{10}{11685817}\zeta^1_{42},\ \
~~~~~~~~~~~~~~~~~~~~~~~~~~~~~~~~~~~~~~~~~~~~~~~~~~~~ 1< c\leq 2, \\
\frac{10}{11685817}\zeta^1_{43},\ \
~~~~~~~~~~~~~~~~~~~~~~~~~~~~~~~~~~~~~~~~~~~~~~~~~~~~ 2< c\leq 3;\\
\frac{10}{11685817}(c-4)^{19} \left(c^4+30 c^3+50 c^2-325 c+95\right),\ \
~~~~ 3< c\leq 4;
\end{cases}
\end{align*}
where
\begin{align*}
\zeta^1_{32}=&2 c^{14}-42 c^{13}+182 c^{12}+1638 c^{11}-30030 c^{10}+234234 c^9-1135134 c^8+3683394 c^7\\&-8237229 c^6+12837825 c^5-13900887 c^4+10248147 c^3\\&-4905992 c^2+1375332 c-171420 \\
\zeta^1_{42}=&-3 c^{23}+138 c^{22}-1518 c^{21}-12397 c^{20}+703087 c^{19}-13863388 c^{18}+176051568 c^{17}\\&
-1584694848 c^{16}+10532925348 c^{15}-53064396088 c^{14}+206513065528 c^{13}\\
&-629711399408 c^{12}+1520203490988 c^{11}-2926140998088 c^{10}+4508152194128 c^9\\&
-5562236749476 c^8+5478760790976 c^7-4275619068336 c^6+2608364956736 c^5\\&-1217086784606 c^4
+419360473840 c^3-100537883820 c^2+14974716720 c-1043516620\\
\zeta^1_{43}=&3 c^{23}-138 c^{22}+1518 c^{21}+30107 c^{20}-1411487 c^{19}+27726776 c^{18}-352103136 c^{17}\\&
+3169389696 c^{16}-20825596836 c^{15}+98921176376 c^{14}-309607140920 c^{13}\\&
+319577156080 c^{12}+2998895483220 c^{11}-23166032264760 c^{10}+98489523542960 c^9\\&
-300901360559844 c^8+702625814387904 c^7-1274022686388144 c^6+1787284754851904 c^5\\&
-1906147044797534 c^4+1494846453598960 c^3-812841418001580 c^2\\&+273722184690480 c-42989308742860
\end{align*}
Case \uppercase\expandafter{\romannumeral3}\::\:\: $\alpha=1$, $\beta=2$
 \begin{align*}
\mathbb{P}(c,1,2,2)&=
\begin{cases}
\frac{10c^5 \left(c^4-12 c^3+54 c^2-84 c+42\right)}{7}, ~~~~~~~~~~~~~~~~~~~~~~~~~~~~~~~~~~ 0\leq c\leq 1, \\
\frac{10(2-c)^7 \left(c^2+2 c-2\right)}{7},\ \
~~~~~~~~~~~~~~~~~~~~~~~~~~~~~~~~~~~~~~~~~ 1< c\leq 2;
\end{cases}\\
\mathbb{P}(c,1,2,3)&=
\begin{cases}
\frac{5c^{11} \left(c^6-34 c^5+476 c^4-2992 c^3+9044 c^2-12376 c+6188\right)}{2431}, ~~~~~~~~~~~ 0\leq c\leq 1,\\
\frac{5\zeta^2_{32}}{2431},\ \
~~~~~~~~~~~~~~~~~~~~~~~~~~~~~~~~~~~~~~~~~~~~~~~~~~~~~~~~ 1< c\leq 2, \\
\frac{5(c-3)^{14} \left(c^3+8 c^2-7 c-10\right)}{2431},~~~~~~~~~~~~~~~~~~~~~~~~~~~~~~~~~~~~~ 2< c\leq 3;
\end{cases}\\
\mathbb{P}(c,1,2,4)&=\begin{cases}
\frac{14c^{19} \left(c^8-72 c^7+2223 c^6-33930 c^5+280800 c^4-1291680 c^3+3256110 c^2-4144140 c+2072070\right)}{455746863}, ~~~ 0\leq c\leq1,\\
\frac{14}{455746863}\zeta^2_{42},\ \
~~~~~~~~~~~~~~~~~~~~~~~~~~~~~~~~~~~~~~~~~~~~~~~~~~~~~~~~~~~~~~~~~~~~~~~~ 1< c\leq 2, \\
\frac{14}{455746863}\zeta^2_{43},\ \
~~~~~~~~~~~~~~~~~~~~~~~~~~~~~~~~~~~~~~~~~~~~~~~~~~~~~~~~~~~~~~~~~~~~~~~~ 2< c\leq 3,\\
\frac{14}{455746863}(4-c)^{23} \left(c^4+20 c^3+15 c^2-166 c+58\right),\ \
~~~~~~~~~~~~~~~~~~~~~~~~ 3< c\leq 4.
\end{cases}
\end{align*}
where
\begin{align*}
\zeta^2_{32}=&-2 c^{17}+68 c^{16}-952 c^{15}+5984 c^{14}-6188 c^{13}-210392 c^{12}+2165800 c^{11}-12602304 c^{10}\\
&+50803038 c^9-148864716 c^8+321854676 c^7-514965360 c^6+606448752 c^5\\
&-517823264 c^4+311355748 c^3-124876832 c^2+29971221 c-3254970\\
\zeta^2_{42}=&-3 c^{27}+216 c^{26}-6669 c^{25}+101790 c^{24}-637650 c^{23}-5920200 c^{22}+216087300 c^{21}\\
&-3344913000 c^{20}+36142228980 c^{19}-297557145600 c^{18}+1923619208940 c^{17}\\
&-9918848071080 c^{16}+41224381129620 c^{15}-139176493635600 c^{14}+383891999309100 c^{13}\\
&-868669502439960 c^{12}+1616404010663520 c^{11}-2475012726838080 c^{10}\\
&+3114293148449340 c^9-3208489645818600 c^8+2688680200441950 c^7\\
&-1813643235261000 c^6+969395892583950 c^5-400947964192620 c^4+123695368658550 c^3\\
&-26784799138656 c^2+3630982483332 c-231837186488\\
\zeta^2_{43}=&3 c^{27}-216 c^{26}+6669 c^{25}-101790 c^{24}+432900 c^{23}+15715440 c^{22}-441942930 c^{21}\\
&+6702258420 c^{20}-72290674170 c^{19}+595114291200 c^{18}-3835663834860 c^{17}\\
&+19460223006120 c^{16}-76541709247380 c^{15}+219451115362320 c^{14}\\
&-348837579341100 c^{13}-516681174695400 c^{12}+6258220890948000 c^{11}\\
&-26964761415134400 c^{10}+80427656700697020 c^9-184373541679041000 c^8\\
&+334653488151904350 c^7-483245113519139400 c^6+550261669854516750 c^5\\
&-484028759772387180 c^4+317470979938360950 c^3-146172645886501728 c^2\\
&+42141647696842116 c-5722716024060344
\end{align*}

\subsection*{Acknowledgements}
~~~~The financial support of the Macau Science and Technology Development Fund under grant number FDCT 130/2014/A3, and grant number FDCT 023/2017/A1
is gratefully acknowledged. We would also like to thank the University of Macau for generous support: MYRG 2014-00011 FST, MYRG 2014-00004 FST.

\bibliographystyle{plain}

\begin{thebibliography}{16}
\bibitem{Adler} {M. Adler, P}. Van Moerbeke, Hermitian, symmetric and symplectic random ensembles: PDEs for the distribution of the spectrum,
Ann. Math. {\bf 153} (2001), 149--189.
\bibitem{BasorChen2009} {E. Basor and Y. Chen}, Painlev\'{e} V and the distribution function of a discontinuous linear statistic in the Laguerre unitary ensembles, J. Phys. A: Math. Theor. {\bf 42} (2009), 035203 (18pp).
\bibitem{BasorChen2005} {E. Basor and Y. Chen}, Perturbed Hankel determinants, J. Phys, A: Math. Gen. {\bf 38} (2005), 10101-10106.
\bibitem{BasorChen2001} {E. Basor, Y. Chen and H. Widom}, Determinants of Hankel matrices, J. Funct. Anal. {\bf 179} (2001), 214-234.
\bibitem{BasorChenEhrhardt2010} {E. Basor, Y. Chen and T. Ehrhardt}, Painlev\'{e} V and time-dependent Jacobi polynomials, J. Phys. A: Math. Theory. {\bf 43}
(2010), 015204 (25pp).
\bibitem{Bauldry1990} {W. Bauldry}, Estimate of the asymmetric Freud polynomials on the real line, J. Approx. Theory. {\bf 63} (1990), 225--237.
\bibitem{B1}
G. Blower, Almost sure weak convergence for the generalized orthogonal ensemble, J. Statist. Phys {\bf 105} (2001), 309-335.
\bibitem{B2}
G. Blower, Displacement convexity for the generalized orthogonal ensemble, J. Statist. Phys {\bf 116} (2004), 1359-1387.
\bibitem{Bonan1990} {S. Bonan, D. S}. Clark, Estimates of the Hermite and Freud polynomials, J. Approx.Theory. {\bf 63} (1990), 210-224.
\bibitem{Boyd} S. Boyd and L. Vandenberghe, {\em Convex Optimization}, 7th edition Cambridge University Press, 2004.\par
\bibitem{ChenChenJMP2015} {M. Chen and Y. Chen}, Singular linear statistics of the Laguerre unitary ensemble and Painlev\'{e} III. Double scaling analysis, J. Math. Phys. {\bf 56} (2015), 063506 (14pp).
\bibitem{ChenIts2010} {Y. Chen and A. Its}, Painlev\'{e} III and a singular linear statistics in Hermitian random matrix ensembles, I, J. Approx.
Theory {\bf 162} (2010), 270--297.
\bibitem{ChenZhang2010} {Y. Chen and L. Zhang}, Painlev\'{e} VI and the unitary Jacobi ensembles,
Stud. Appl. Math. {\bf 125} (2010), 91--112.
\bibitem{ChenIsmail2005} {Y. Chen and M. Ismail}, Jacobi polynomials from campatibilty conditions, Proc. Amer. Math. Soc. {\bf 133} (2005), 465-472.
\bibitem{ChenIsmail1997} {Y. Chen and M. Ismail}, Ladder operators and differential equations for orthogonal polynomials, J. Phys. A: Math. Gen. {\bf 30} (1997), 7817-7829.
\bibitem{ChenMcKay2012} {Y. Chen and M. R. McKay},  Coulumb fluid, Painlev\'{e} transcendents, and the information theory of MIMO systems,
IEEE Trans. Inf. Theory {\bf 58} (2012), 4594--4634.
\bibitem{ChenFeigin2006} {Y. Chen and M V Feigin}, Painlev\'{e} IV and degenerate Gaussian unitary ensembles, J. Phys. A: Math. Gen. {\bf 39} (2006),
12381--12393.
\bibitem{ChenLawrence1998} {Y. Chen and N. Lawrence}, On the linear statistics of Hermitian random matrices, J. Phys. A: Math. Gen.  {\bf 31} (1998), 1141--1152.
\bibitem{DKM}
P. Deift, T. Kriecherbauer, K.T.-R McLaughlin, New results on the equilibrium measure for the logarithmic potentials in the presence of an external field, J. Approx. Theory {\bf 95} (1998), 388-475.
\bibitem{Dyson} {F. J. Dyson}, Statistical theory of the energy levels of complex systerms \textrm{I}-\textrm{III}, J. Math. Phys. {\bf 3} (1962), 140-175.
\bibitem{GradshteynRyzhik2007} {I. S. Gradshteyn and I. M. Ryzhik}, {\em Table of Integrals, Series, and Products 8th ed.}, Academic Press, 2014.
\bibitem{Toda1993} {L. Haine and E. Horozov}, Toda orbits of Laguerre polynomials and representations of the Virasoro algebra, Bull. Sci. Math.{\bf 17} (1993), 485-518.
\bibitem{JimboMiwa1981} {M. Jimbo and T. Miwa}, Monodromy perserving deformation of linear ordinary differential equations with rational coefficients. II,
Physica D {\bf 2} (1981), 407-448.
\bibitem{KRRR} {J. P. Keating, B. Rodgers, E. Roditty-Gershon, Z. Rudnick,} Sums of divisor functions in $F_q[t]$ and matrix integrals, Mathematische Zeitschrift, DOI 10.1007/s00209-017-1884-1.
\bibitem{LP}
M. Ledoux and I. Popescu, Mass transportation proofs of free functional inequalities and free Poincare inequalities, J. Funct. Anal. {\bf 257} (2009), 1175-1221.
\bibitem{Magnus} {A. P. Magnus}, Painlev\'{e}-type differential equations for the recurrence coefficients of semi-classical orthogonal
polynomials, {Journal of Computational and Applied Mathematics} {\bf 57} ({1995}), {215--237}.
\bibitem{Masuda} T. Masuda, Y. Ohta and K. Kajiwara, A determinant formula for a class of rational solutions of the Painlev\'e V equation, Nagoya Math. J. {\bf 168} (2002), 1-25.\par
\bibitem{Mehta2006} {M. L. Mehta}, {\em Random Matrices 3rd ed.}, Elsevier (Singapore) Pte Ltd., Singapore, 2006.
\bibitem{MCChen} {C. Min, Y. Chen}, Linear statistics of matrix ensembles in classical background, J. Mathematical Methods in the Applied Sciences, 2016.
\bibitem{MCChen2016} {C. Min, Y. Chen}, On the variance of linear statistics of Hermitian random matrices, Acta Physica Polonica B. {\bf 47} (2016), 1127.
\bibitem{Moser1975} {J. Moser}, Finitely many mass points on the line under the influence of an exponential potential -- an integrable system, in Dynamical Systems, Theory and Applications, Springer Berlin Heidelberg, (1975), 467-497.
\bibitem{Okamoto} {K. Okamoto}, On the $\tau$-function of the Painlev\'{e} equations, Physica D {\bf 2} (1981), 525-535.
\bibitem{ST}
E. B. Saff and V. Totik, {\em Logarithmic Potentials with External Fields}, Springer, 1997.
\bibitem{LvChen} {L. Shulin, Y. Chen}, The largest eigenvalue distribution of the Laguerre unitary ensemble, Acta Mathematica Scientia {\bf 37}.2 (2017), 439-462.
\bibitem{sogo1993} {K. Sogo}, Time dependent orthogonal polynomials and theory of solition. Applications to matrix model, vertex model and level statistics, J. Phys. Soc. Japan. {\bf 62}  (1993), 1887-1894.
\bibitem{SO}
A. Stuart and J.K. Ord, {\em Kendall's Advanced Theory of Statistics: Volume 1: Distribution Theory 5th ed.} , The Universities Press, Belfast, 1987.
\bibitem{SW}
E.M. Stein and G.L. Weiss, {\em Introduction to Fourier analysis on Euclidean spaces}, Princeton University Press, 1971.
\bibitem{Szego1939} {G. Szeg\"{o}}, {\em Orthogonal Polynomials}, {AMS}, {New York}, {1939}.
\bibitem{TracyWidom1994} {C. A. Tracy and H. Widom}, Fredholm determinants, differential equations and matrix models,Commun. Math. Phys. {\bf 163} (1994), 33--72.
\bibitem{MTsuji} {M. Tsuji}, {\em Potential Theory in Modern Function Theory}. Tokyo, Japan: Maruzen, 1959.
\bibitem{Weyl1946} {H. Weyl}, {\em The Classical Groups 2rd ed.}, Princeton University Press, 1946.
\bibitem{Witten1991} {E. Witten}, Two-dimensional gravity and intersection theory on moduli space, Surveys in Differential Geometry 1, A
supplement to the Journal of Differential Geometry, edited by C. C. Hsiung and S. T. Yau (1991), 243--310.




\end{thebibliography}

\end{document}